\documentclass[12pt]{article}
\usepackage{cancel}
\usepackage{tabularx}
\usepackage{tikz}
\usepackage{relsize}
mm \textheight=8.9in

\usepackage{amssymb}
\usepackage{amsmath}
\usepackage{amsthm}
\usepackage{color}

\usepackage{tikz}
\usetikzlibrary{patterns,snakes}

\usepackage{tikz-cd}

\usepackage{array,multirow}

\renewcommand{\i}{\mathrm{i}}
\newcommand{\br}[3]{{$#1$}$\lower4pt\hbox{$\tp\atop\raise4pt \hbox{$\scriptscriptstyle{#2}$}$} ${$#3$}}
\newcommand{\tw}[3]{{$#1$}${\,\scriptscriptstyle {#2}}\atop\raise9pt\hbox{$\scriptstyle\tp$} ${$#3$}}
\newcommand{\ttps}[2]{{#1}\raise5pt\hbox{$\lower12pt\hbox{$\scriptstyle\tp$}\atop \lower0pt\hbox{$\tilde\;$}$}\raise4.5pt\hbox{${\scriptstyle{#2}}$}}
\newcommand{\st}[1]{\mbox{${\,\scriptscriptstyle {#1}}\atop\raise5.5pt\hbox{$*$}$}}

\newcommand{\rd}[1]{\mbox{${\,\scriptscriptstyle {#1}}\atop\raise5.5pt\hbox{$\bullet$}$}}
\newcommand{\rt}[1]{\otimes_\chi}
\newcommand{\lt}[1]{\mbox{${\,\scriptscriptstyle {#1}}\atop\raise5.5pt\hbox{$\ltimes$}$}}
\newcommand{\btr}{\raise1.2pt\hbox{$\scriptstyle\blacktriangleright$}\hspace{2pt}}
\newcommand{\btl}{\raise1.2pt\hbox{$\scriptstyle\blacktriangleleft$}\hspace{2pt}}

\newcommand{\lcr}{\raise1.0pt \hbox{${\scriptstyle\rightharpoonup}$}}
\newcommand{\rcr}{\raise1.0pt \hbox{${\scriptstyle\leftharpoonup}$}}

\newcommand{\ttp}{{\lower12pt\hbox{$\tp$}\atop \hbox{$\tilde\;$}}}

\newcommand{\id}{\mathrm{id}}

\renewcommand{\i}{\mathrm{i}}

\newcommand{\Dc}{\mathcal{D}}
\newcommand{\Bc}{\mathcal{B}}
\newcommand{\Ac}{\mathcal{A}}

\newcommand{\Ru}{\mathcal{R}}

\newcommand{\Uc}{\mathcal{U}}

\newcommand{\Kc}{\mathcal{K}}

\newcommand{\C}{\mathbb{C}}
\newcommand{\Z}{\mathbb{Z}}

\newcommand{\N}{\mathbb{N}}

\newcommand{\tp}{\otimes}

\newcommand{\zt}{\zeta}

\newcommand{\U}{U}

\newcommand{\ve}{\varepsilon}
\newcommand{\gm}{\gamma}
\newcommand{\dt}{\delta}

\newcommand{\op}{\oplus}
\newcommand{\la}{\lambda}

\newcommand{\End}{\mathrm{End}}

\newcommand{\Span}{\mathrm{Span}}
\newcommand{\Aut}{\mathrm{Aut}}

\newcommand{\rk}{\mathrm{rk}}

\newcommand{\Rm}{\mathrm{R}}

\newcommand{\La}{\Lambda}

\newcommand{\g}{\mathfrak{g}}
\renewcommand{\b}{\mathfrak{b}}
\renewcommand{\k}{\mathfrak{k}}
\newcommand{\h}{\mathfrak{h}}

\newcommand{\s}{\mathfrak{s}}

\renewcommand{\o}{\mathfrak{o}}
\newcommand{\n}{\mathbf{n}}

\newcommand{\m}{\mathbf{m}}
\newcommand{\eps}{\epsilon}

\newcommand{\nn}{\nonumber}
\newcommand{\p}{\mathfrak{p}}
\renewcommand{\l}{\mathfrak{l}}
\renewcommand{\c}{\mathfrak{c}}

\newcommand{\si}{\sigma}
\newcommand{\al}{\alpha}
\renewcommand{\t}{\mathfrak{t}}
\newcommand{\bt}{\beta}

\newcommand{\be}{\begin{eqnarray}}
\newcommand{\ee}{\end{eqnarray}}

\newtheorem{thm}{Theorem}[section]
\newtheorem{propn}[thm]{Proposition}
\newtheorem{lemma}[thm]{Lemma}
\newtheorem{corollary}[thm]{Corollary}
\newtheorem{conjecture}{Conjecture}

\newtheorem{remark}[thm]{Remark}
\newtheorem{definition}[thm]{Definition}

\newcount\prg

\newcommand{\parag}{\advance\prg by1 {\noindent\bf\thesection.\the\prg\hspace{6pt}}}

\begin{document}

\title{Quantum super-spherical pairs}
\author{D. Algethami${}^{\dag,\ddag}$, A. Mudrov${}^{\dag,\sharp}$, V. Stukopin${}^\sharp$
\vspace{10pt}\\
\small ${\dag}$ University of Leicester, \\
\small University Road,
LE1 7RH Leicester, UK,
\vspace{10pt}\\
\small ${\ddag}$ Department of Mathematics, College of Science,\\
\small
 University of Bisha, P.O. Box 551, Bisha 61922, Saudi Arabia,
\vspace{10pt}\\
\small
${\sharp}$ Moscow Institute of Physics and Technology,\\
\small
9 Institutskiy per., Dolgoprudny, Moscow Region,
141701, Russia,
\vspace{10pt}\\
\small
 e-mail: daaa3@leicester.ac.uk,  mudrov.ai@mipt.ru, stukopin.va@mipt.ru
}

\maketitle

\begin{abstract}
We introduce quantum super-spherical pairs  as coideal subalgebras in general linear and orthosymplectic  quantum supergroups. These subalgebras play a role of isotropy subgroups for
matrices solving the $\mathbb{Z}_2$-graded reflection equation.
They generalize quantum (pseudo)-symmetric pairs of Letzter-Kolb-Regelskis-Vlaar.

\end{abstract}

\\


\newpage
\tableofcontents

\section{Introduction}


This paper is devoted to a $\Z_2$-graded generalization  of quantum symmetric pairs pioneered by Letzter \cite{Let} and  developed to a deep theory of
quantum symmetric spaces, \cite{K,BK,RV,AV}.
The classical supergeometry is an established field of mathematics \cite{Ma}, \cite{VMP} motivated by profound  applications of supersymmetry in quantum physics \cite{Dir,FP}.  That refers to the concept of symmetric  and, more generally, spherical spaces \cite{Sh,Sh1}. Quantum symmetric pairs stem from  the
Reflection Equation (RE), which plays for them a similar fundamental role as the Yang-Baxter Equation for quantum groups. Quantum supergroups have been in research focus from the very
birth of quantum groups and are now well understood \cite{Y1}. It seems natural to unify these precursors in a theory of quantum super-symmetric pairs.
Such an attempt was made in \cite{Shen}
for a special case of general linear supergroups following ideas of  \cite{KY}.
The current paper is a step further to a general theory of quantum super-spherical pairs. We have also to mention \cite{ShenWang}
addressing a similar topic that appeared shortly after our arXiv preprint.

We are not aiming at a comprehensive exposition of this broad area but focus on the special case of classical (basic) matrix  Lie superalgebras
and confine  ourselves with  a special symmetric polarization of the root system. Unlike \cite{Shen,ShenWang} we do not appeal to
the Radford-Majid bosonization of quantum supergroups \cite{Y2} but stay within the category of Hopf superalgebras.
Our approach is close to \cite{RV} and based on  classification of  graded (generalized)   Satake diagrams, for a fixed Borel subalgebra in $\g$.
However our logic is rather inverse  comparing to \cite{RV}: we start with transparent quasiclassical conditions on a subalgebra $\k$ that allows
for Letzter's quantization  and arrive at an involutive automorphism of the Cartan matrix,
which is a starting point  of  \cite{RV}. Such an automorphism is subordinate to a subset of simple roots of $\g$. Together they define a generalized Satake diagram:
a combinatorial datum underlying $\k$.
We then solve the RE for a special class of $\k$ and relate the K-matrices to the corresponding coideal subalgebras as their quantum isotropy subgroups.

Recall that a subalgebra $\k$ in a simple Lie algebra $\g$ is called spherical if there is a Borel subalgebra $\b\subset \g$ such that
$\k+\b=\g$. The pair $(\g,\k)$ is the localization of a classical homogeneous spherical manifold. Such manifolds generalize symmetric spaces and feature similar nice representation theoretical properties \cite{VK}.
Solutions to non-graded RE   for standard $U_q(\g)$ deliver  classical points on spherical manifolds, where
the Poisson bracket vanishes \cite{AlM}. Their isotropy Lie algebras enter spherical pairs $(\g,\k)$ with the total Lie algebra $\g$.

We adopt a similar definition for super-spherical pairs assuming that both $\g$ and  $\k$ are graded. As different Borel subalgebras in $\g$
are generally not conjugated and even not isomorphic (it is an easy exercise already for $\o\s\p(2|2)$), this definition depends on the choice of $\b$. We study spherical pairs that are quantizable along the similar lines as their non-graded (pseudo) symmetric analogs \cite{Let,K,RV}.

\subsection{Current work}
 We define a classical pseudo-symmetric pair $(\g,\k)$ via an involutive automorphism $\tau$ of the Dynkin diagram  $D_\g$  subordinate to a certain subdiagram $D_\l\subset D_\g$, cf. Definition \ref{triple}.
Here $\l$ is an analog of the semisimple Lie subalgebra in $\k$ whose roots are fixed by the symmetry in the ordinary symmetric pair $(\g,\k)$.
The key novelty here is a substitute for the longest Weyl group element $w_\l$ of $\l$, which is present in the non-graded case and which
plays a crucial role in the non-graded theory. For an admissible subset $\Pi_\l\subset \Pi$ of
simple roots specified by Definition \ref{admissible}, we introduce an operator $w_\l$ that features basic properties of the longest Weyl group element.
In particular, it is an even involution that preserves the root systems  of $\l$ and $\g$ with their weight lattices and flips the highest and lowest weights of
basic $\l$-submodules in $\g\ominus \l$ generating $\k$ over $\l$ and a certain subset $\t$ of the Cartan subalgebra $\h\subset \b$.
Such an operator can be defined if the total grading of $\g$ induces a special "symmetric" grading on $\l$.
The relation of $\tau$ with $\l$ consists in the requirement $\tau|_{\Pi_\l}=-w_\l$.

We prove that subalgebras $\k\subset \g$ constructed this way are spherical (Proposition \ref{ps-sym=sup-sph}) and quantizable as coideal subalgebras
in the  quantum supergroup $U_q(\g)$ (Theorem \ref{thm}). They can be conveniently parameterized by decorated Dynkin diagrams associated
with the triples $(\g,\l, \tau)$,  which amount to graded (generalized) Satake diagrams upon filtering through a set of selection rules
in Section \ref{SecDDD-SR}.

For each decorated Dynkin diagram, $\k$ is generated over $\l+\t$ by $x_\al=e_\al+c_\al f_{\tilde \al}+\grave{c}_\al u_\al$, 
$\al \in \bar \Pi_\l=\Pi\backslash \Pi_\l$, where
the positive root $\tilde \al$ is determined by $\tau$,
see (\ref{gen-spher-superpairs}). Here $u_\al\in \h$ is an element centralizing $\l$ for even  $\al\in \bar \Pi_\l$  orthogonal to $\Pi_\l$,
and scalars  $c_\al\in \C^\times$, $\grave{c}_\al\in \C$  are called mixture parameters.
Thus every  diagram gives rise to a whole  family of spherical subalgebras $\k\subset \g$.

There are two problems arising in connection with such a combinatorial description of $\k$: a) it needs to be checked if the
subalgebra $\k$ is proper, that is,
$\k\varsubsetneq \g$, and b) one has to check if $\k$ is  related with a non-trivial K-matrix.
The case  $\k=\g$ should be regarded as trivial and  not interesting. Note that, while the presence of $K$-matrix is of prime interest by itself, it is a sufficient condition for a) meaning that $\k$ has a bigger supply of matrix invariants than $\g$.
We give a classification of graded Satake diagrams within the fixed grading by discarding decorated diagrams leading to trivial  pairs with $\k=\g$ at any choice
of mixture parameters.
Our selection rules reduce  to forbidding
subdiagrams (\ref{RVSR}-\ref{D-TAIL}) to appear in a decorated Dynkin diagram. They turn out to be  more intricate than the non-graded selection rules, which involve  only
one subdiagram (\ref{RVSR}), see \cite{RV} for details.

Graded generalized Satake diagrams are listed in Section \ref{Sec-Gr-Satake}. Their initial classification may be conducted by the shape of the diagram. By this
we mean the non-graded decorated diagram obtained by throwing  the grading away. Diagrams (\ref{GL-I})--(\ref{C-type-diag}) whose shape is an admissible non-graded generalized
Satake diagrams  are said to be of type I. The remaining diagrams of type II are "essentially graded" and comprise (\ref{ANOM-GL}), (\ref{ANOM-OSP})
and (\ref{C-type-diag0}). All they have odd roots in $\bar \Pi_\l$.

We present K-matrices for type I in Theorem \ref{s-graded} thus proving that the corresponding pairs $(\g,\k)$ are proper.
With regard to type II we conjecture that they produce proper $\k$ at least for some values of mixture parameters defining
$\k$. For the special case of (\ref{C-type-diag0}) that holds due to the K-matrix  presented in Theorem \ref{s-graded}.
We expect that diagrams  (\ref{ANOM-GL}) and (\ref{ANOM-OSP}) lead to solutions of a twisted version (\ref{REtw}) of RE relative to the matrix super-transposition.
Their K-matrices are believed to be (\ref{GL-left}) and (\ref{GL-right}).
Those diagrams from (\ref{ANOM-OSP}) associated with non-trivial flip of the Satake diagram are related with RE twisted by the outer automorphism (\ref{REtw1}).
Their conjectured K-matrices are explicated in (\ref{K-white-tailed-twisted}).
With regard to the remaining part of diagrams  (\ref{ANOM-OSP}), we guess their K-matrices (\ref{K-black-tailed}), (\ref{K-white-tailed}), and (\ref{K-half-tailed})
by studying the simplest examples.
This makes us believe that all graded Satake diagrams give rise to non-trivial spherical pairs associated with an appropriate version of RE.
This discussion will be a matter of a forthcoming study.

Solutions to the RE for the spherical pairs of type I are given in Section \ref{SecKmat}, Theorem \ref{s-graded}.
They turn out to have a similar shape as in the non-graded case, with a certain restriction on the location of the odd root
for the ortho-symplectic $\g$, see Corollary \ref{Levi-Gr-Boxes}.

For each Satake diagram, we quantize the universal enveloping algebra  $U(\k)$ of the relative spherical subalgebra $\k\subset \g$
as a coideal subalgebra $U_q(\k)\subset U_q(\g)$, cf. Theorem \ref{thm}. Therein we relate the K-matrices found
in Section \ref{SecKmat} to corresponding $U_q(\k)$.

\section{Preliminaries}

This section contains  a general  description of quantum supergroups deforming the universal enveloping algebras of general linear
 and orthosymplectic Lie superalgebras. We explicate their natural representations on  the graded vector space $ \C^{N|2\m}$ along
with the graded R-matrices in a symmetric grading with minimal number of odd simple roots.

\subsection{Quantum Supergroup $U_q (\g)$}
For a textbook on quantum groups, the reader is referred to \cite{ChP}. In our exposition of quantum supergroups, we follow \cite{KT1}.

An algebra $\Ac$ is called superalgebra if it is $\Z_2$-graded: that is, $\Ac=\Ac_0\oplus \Ac_1$, and $\Ac_i\Ac_j\subset \Ac_{i+j\!\!\mod 2}$.
Elements of $\Ac_0$ are called even and elements of $\Ac_1$ are called odd.

Given two  graded associative algebras $\Ac$ and $\Bc$  their tensor product $\Ac\tp \Bc$ is a graded algebra too.
The multiplication on homogeneous elements is determined by the rule
$$
(a_1\tp b_1)(a_2\tp b_2)=(-1)^{|a_2||b_1|}a_1 a_2\tp b_1b_2.
$$
where $|a|\in \{0,1\}$ stands for  the degree of $a$.

In this section, we  recall a general definition of quantum supergroup $U_q (\g)$ associated with a graded Lie superalgebra $\g$
possessing  a set of Chevalley-like generators. It is a quasi-triangular Hopf superalgebra
with comultiplication ranging in the graded commutative tensor square of $U_q(\g)$.
The antipode in $U_q(\g)$ is a graded anti-automorphism: $\gm(ab)=(-1)^{|a||b|}\gm(b)\gm(a)$ for all homogeneous $a,b\in U_q(\g)$.

Let $\g$ be a finite-dimensional complex  Lie superalgebra associated with a polarized root system $\Rm=\Rm^-\cup \Rm^+$ of rank $n$  with a basis $\Pi=\{\alpha_i: i \in[1,n]=I\}\subset \Rm^+$ of the simple roots and even Cartan subalgebra $\h\subset \g$. Let  $A=(a_{ij})_{1\leq i,j\leq n}$, denote the symmetrizable Cartan matrix and $(-,-)$ the corresponding
non-degenerate symmetric bilinear form on the dual vector space $\h^*$.

Fix  $q\in \C^\times$ to be not a root of unity. Define quantum integers   by setting
$$[n]_q=q^{n-1}+q^{n-3}+\dots+q^{-n+3}+q^{-n+1}=\frac{q^n-q^{-n}}{\omega}$$
for $n\in \mathbb{Z}$,
with the notation, $\omega=q-\bar q$, $\bar q=q^{-1}$, used  throughout the text.
For each simple root $\al\in \Pi$ define
$$
 q_{\alpha}=\begin{cases}
 q, & \mbox{if } (\alpha,\alpha)=0, \\
 q^\frac{(\alpha,\alpha)}{2}, & \mbox{if } (\alpha,\alpha)\neq0.
 \end{cases}
$$
Denote by $\varkappa\subset I$
the subset indexing even simple roots of $\g$.
\begin{definition}\cite{KT1}
  The quantum supergroup $U_q (\g)$ is a complex unital associative superalgebra   generated by
 $ e_{\pm\alpha_i},$ and $q^{\pm h_{\alpha_i}}$ with grading
$$|e_{\pm{\alpha_i}}|=\begin{cases}
 0, & \mbox{if } i\in \varkappa\subset I  \\
 1, & \mbox{if } i\notin \varkappa.
 \end{cases},
 \quad
|q^{\pm h_{\alpha_i}}|=0,\ \forall\ i\in I,
$$ such that the following relations are satisfied:
\begin{enumerate}
  \item [(i)]	$q^{h_{\alpha_i} } q^{-h_{\alpha_i} }=q^{-h_{\alpha_i} } q^{h_{\alpha_i} }=1,$ $ q^{h_{\alpha_j} } q^{h_{\alpha_i} }=q^{h_{\alpha_i}} q^{h_{\alpha_j} }$,
\item [(ii)]	$q^{h_{\alpha_i} }  e_{\pm{\alpha_j} }  q^{-h_{\alpha_i} }=q^{\pm(\alpha_i,\alpha_j ) } e_{\pm\alpha_j }$,

\item [(iv)] $[e_{\alpha_i },e_{-\alpha_j}]=e_{\alpha_i }e_{-\alpha_j}-(-1)^{|e_{-\alpha_j }||e_{\alpha_i }|}e_{-\alpha_j }e_{\alpha_i }=\delta_{ij}[h_{\al_i}]_{q_{\al_i}}$,\item [(v)] $(ad_{q'} e_{\pm{\alpha_i}})^{v_{ij}} e_{\pm{\alpha_j}}=0,\quad i\neq j,\ q'=q, \bar q,$ where
 $$(ad_{q'} e_{\al_i}) x=e_{\al_i}x-(-1)^{|e_{\al_i}||x|}(q')^{(\alpha_i,wt(x))}x e_{\al_i},\quad x\in U_q(\g),$$
 $$ v_{ij}=\begin{cases}
                 1, & \mbox{if } (\alpha_i,\alpha_i)=(\alpha_i,\alpha_j)=0, \\
                 2, & \mbox{if } (\alpha_i,\alpha_i)=0,\ (\alpha_i,\alpha_j)\neq0,  \\
                 1-\frac{2(\alpha_i,\alpha_j)}{(\alpha_i,\alpha_i)} , & \mbox{if }(\alpha_i,\alpha_i)\neq0.
               \end{cases}$$
 \end{enumerate}
\end{definition}
\noindent
The left equality in (iv) reminds that the commutator is understood in the graded sense, while the right one imposes a relation of $U_q(\g)$.

We will work with general linear Lie superalgebras too.
To that end, we need to extend the previous definition of  special linear $\g$ by adding Cartan generators
 $\{q^{h_{\zeta_i}} \}_{i=1}^{n+1}$ with commutation relations
$$q^{h_{\zeta_i} }  e_{\pm{\alpha_j} }  q^{-h_{\zeta_i} }=q^{\pm(\zeta_i,\alpha_j ) } e_{\pm\alpha_j }.$$
We then proceed  further by setting $q^{h_{\al_i} } =q^{h_{\zt_i} } q^{-h_{\zt_{i+1}} }
$.

A Hopf superalgebra structure is fixed by  comultiplication defined on the generators
$e_i=e_{\alpha_i }$, $f_i=e_{-\alpha_i }$, and $q^{\pm h_i}=q^{\pm h_{\al_i}}$ as
$$
\Delta(e_i)= q^{h_i }\otimes e_i+e_i\otimes 1,\quad \Delta(f_i)=f_i\otimes q^{-h_i }+1\otimes f_i,\quad \Delta(q^{\pm h_i })=q^{\pm h_i}\otimes q^{\pm h_i}.
$$
Mind that the tensor product is supercommutative: $a\tp b=(-1)^{|a||b|}b\tp a$ for all homogeneous $a$ and $b$.

The counit $\eps$ is defined as the homomorphism $U_q(\g)\to \C$ that vanishes on all $e_i$, $f_i$ and returns $1$ on $q^{\pm h_i}$.
The antipode $\gamma$ can be readily evaluated on the generators as
$$ \gamma(q^{\pm h_{i} })=q^{\mp h_{i} },\quad  \gamma(e_{i})=-q^{-h_{i} }e_{i},\quad \gamma(f_{i})=-f_{i} q^{h_{i} }$$
and extended as a graded anti-automorphism to entire $U_q(\g)$. This supercoalgebra structure extends to the Cartan generators
of the quantum general linear quantum supergroup  in the obvious way.

We denote by  $U_q(\h),\ U_q(\g_+),\ \mathrm{and}\ \U_q(\g_-)$ the $\Z_2$-graded $\C$-subalgebras of $U_q (\g)$ generated by $q^{\pm h_{i} },\ e_{i},$ and $f_{i}$  respectively.

Simple root vectors enter the set of generators of $U_q(\g)$, while composite root vectors are not given for granted. For
 quantum groups they can be constructed via the Lusztig braid group action on $U_q(\g)$  \cite{ChP}, which is not applicable to quantum supergroups.
 Still the composite root vectors can be obtained by  a method of Khoroshkin and Tolstoy \cite{KT} based on the concept of normal
 ordering of positive roots, which works in both cases. As a result, root vectors
 $e_\al\in U_q(\g_+)$ and $f_\al\in U_q(\g_-)$ can be defined for each positive $\al\in \Rm^+$.
 They are used for a description of the universal R-matrix of $U_q(\g)$ in \cite{KT}.
 We will need them in Section \ref{SecCoidSubA}
 for construction of coideal subalgebras in $U_q(\g)$.

\subsection{Basic Quantum Supergroups}
\label{SecBasQSG}

Next we specify the root systems of basic (non-exceptional) matrix Lie  superalgebras: general linear $\g\l(N|2\m)$ and
ortho-symplectic  $\o\s\p(N|2\m)$, $\s\p\o(N|2\m)$,
with even $N$ in the latter case.
Although $\o\s\p(2\m|2\n)$ is isomorphic to $\s\p\o(2\n|2\m)$ as Lie superalgebras, they are considered as different triangular polarizations
leading to different quantum supergroups. In particular, they have different graded Dynkin diagrams and different K-matrices, see Section \ref{Sec-Gr-Satake}.

We fix the following symmetric grading on underlying vector space $V = \C^{\m+N+\m} = \C^{N|2\m}$:
$$
|i| =
\left \{
\begin{array}{cccc}
  1, & i\leqslant \m \quad \mbox{or} \quad N+\m<i,\\
  0, & \m<i\leqslant \m+N.
\end{array}
\right.
$$
With respect to this grading, there are two  odd simple roots in $\g\l(N|2\m)$ and $\o\s\p(2|2\m)$, and one odd simple root for $\o\s\p(N|2\m)$, $N\neq 2$,
and $\s\p\o(2\n|2\m)$.
Let $\g=\g_0\oplus \g_1$ be a Lie superalgebra of these types. Denote the Cartan subalgebra by $\h\subset \g_0$.
In the dual vector space $\h^*$, choose a basis  $\{\ve_i\}_{i=1}^{k} \cup \{\delta_i\}_{i=1}^{2\m}$,
where $k=N$ for $\g\l(N|2\m)$ and $k=\n$ for $\o\s\p(2\n|2\m)$, $\o\s\p(2\n+1|2\m)$, and $\s\p\o(2\n|2\m)$.
Introduce a symmetric  bilinear form on $\h^*$ by
$$
(\ve_i,\ve_j)=\delta_{i}^j=-(\delta_i,\delta_j),\ (\ve_i,\delta_j)=0.
$$
In the Dynkin diagrams below, an odd simple root
$\alpha$ is denoted by a black node if $(\alpha,\alpha)\neq0$, and a grey one if $(\alpha,\alpha)=0$. Even simple roots are represented by white nodes.

We denote the sets of even and odd roots with $\Rm_0$ and $\Rm_1$, respectively, so that $\Rm=\Rm_0\cup \Rm_1$.
Additionally, we  divide the root system into positive and negative parts $\Rm^\pm $ with the basis of simple roots $\Pi\subset \Rm^+$.

$\bullet\ \g\l(N|2\m)$
$$
 \begin{cases}
   \Rm^+=\{\ve_i-\ve_{j}\}_{i<j}\cup\{\delta_i-\delta_{j}\}_{i< j}\cup\{\delta_i-\ve_{j}\}_{i\leq \m}\cup\{\ve_{j}-\delta_i\}_{i\geq \m+1},\\
   \Pi=\{\delta_i-\delta_{i+1}\}_{i=1,i\neq \m}^{2\m-1}\cup\{\delta_{\m}-\ve_{1}\}\cup\{\ve_i-\ve_{i+1}\}_{i=1}^{N-1}\cup\{\ve_{N}-\delta_{\m+1}\}.\\
 \end{cases}
 $$
with $\Rm_0=\{\ve_i-\ve_{j}, \delta_i-\delta_{j}\}_{i\neq j}$ and $\Rm_1=\pm\{\ve_i-\delta_{j}\}$.
The Dynkin diagram, which describes the relations among simple roots  and indicates their parity, is

\begin{center}
\begin{picture}(330,30)
\put(0,10){\circle{3}}
\put(30,10){\circle{3}}

\put(80,10){\circle{3}}
\put(110,10){\color{gray}\circle*{3}}
\put(140,10){\circle{3}}

\put(82,10){\line(1,0){26}}
\put(112,10){\line(1,0){26}}

\put(142,10){\line(1,0){10}}
\put(178,10){\line(1,0){10}}
\put(160,10){$\ldots$}
\put(190,10){\circle{3}}
\put(220,10){\color{gray}\circle*{3}}
\put(250,10){\circle{3}}

\put(192,10){\line(1,0){26}}
\put(222,10){\line(1,0){26}}

\put(302,10){\line(1,0){26}}

\put(300,10){\circle{3}}
\put(330,10){\circle{3}}

\put(252,10){\line(1,0){10}}
\put(288,10){\line(1,0){10}}
\put(270,10){$\ldots$}

\put(2,10){\line(1,0){26}}
\put(32,10){\line(1,0){10}}
\put(68,10){\line(1,0){10}}
\put(50,10){$\ldots$}

\put(55,14){\smaller[2]$\delta_{\m-1}-\delta_{\m}$}

\put(-15,14){\smaller[2]$\delta_1-\delta_2$}
\put(15,1){\smaller[2]$\delta_2-\delta_3$}

\put(95,1){\smaller[2]$\delta_{\m}-\ve_1$}
\put(125,14){\smaller[2]$\ve_1-\ve_2$}
\put(167,14){\smaller[2]$\ve_{N-1}-\ve_{N}$}

\put(205,1) {\smaller[2]$\ve_{N}-\delta_{\m+1}$}
\put(230,14) {\smaller[2]$\delta_{\m+1}-\delta_{\m+2}$}

\put(280,1) {\smaller[2]$\delta_{2\m-2}-\delta_{2\m-1}$}
\put(310,14) {\smaller[2]$\delta_{2\m-1}-\delta_{2\m}$}

 \end{picture}
\end{center}

$\bullet\ \o\s\p(2\n+1|2\m)$.
We consider separately two cases: $\n\neq 0$ and $\n=0$.\\
For $\n\neq 0$, we have
 $$
 \begin{cases}
  \Rm^+=\{\ve_i\pm\ve_{j}\}_{i<j}\cup\{\delta_i\pm\delta_{j}\}_{i< j}\cup\{\delta_i \pm \ve_j,\ve_i,2\delta_i,\delta_i\},\\
   \Pi=\{\delta_i-\delta_{i+1}\}_{i=1}^{\m-1}\cup\{\delta_\m - \ve_1\} \cup \{\ve_{i} - \ve_{i+1}\}_{i=1}^{\n-1}\cup \{\ve_\n \}.\\
 \end{cases}
$$
with $\Rm_0=\{\pm \ve_i\pm \ve_j,\pm2\delta_i,\pm \delta_i\pm \delta_j,\pm\ve_i\}_{i\neq j}$, $\Rm_1=\{\pm\delta_i,\pm\delta_i\pm\ve_j\}$,
and  Dynkin diagram
\begin{center}
\begin{picture}(200,30)
\put(0,10){\circle{3}}
\put(30,10){\circle{3}}

\put(80,10){\circle{3}}
\put(110,10){\color{gray}\circle*{3}}
\put(140,10){\circle{3}}

\put(82,10){\line(1,0){26}}
\put(112,10){\line(1,0){26}}

\put(142,10){\line(1,0){10}}
\put(178,10){\line(1,0){10}}
\put(160,10){$\ldots$}
\put(190,10){\circle{3}}

\put(2,10){\line(1,0){26}}
\put(32,10){\line(1,0){10}}
\put(68,10){\line(1,0){10}}
\put(50,10){$\ldots$}

\put(55,14){\smaller[2]$\delta_{\m-1}-\delta_{\m}$}

\put(-15,14){\smaller[2]$\delta_1-\delta_2$}
\put(15,1){\smaller[2]$\delta_2-\delta_3$}

\put(94,1){\smaller[2]$\delta_{\m}-\ve_1$}
\put(125,14){\smaller[2]$\ve_1-\ve_2$}
\put(167,14){\smaller[2]$\ve_{\n-1}-\ve_{\n}$}

\put(220,1) {\smaller[2]$\ve_{\n}$}

\put(220,10){\circle{3}}

\put(191,8.5){\line(1,0){24}}
\put(191,11.5){\line(1,0){24}}
\put(211,7){$>$}
 \end{picture}
\end{center}
 In the case of $\n=0$, we have
  $$
 \begin{cases}
  \Rm^+=\{\delta_i\pm\delta_{j}\}_{i< j}\cup\{2\delta_i,\delta_i\},\\
   \Pi=\{\delta_i-\delta_{i+1}\}_{i=1}^{\m-1}\cup\{\delta_\m \}.\\
 \end{cases}$$
with   $\Rm_0=\{\pm2\delta_i,\pm \delta_i\pm \delta_j\}_{i\neq j}$, $\Rm_1=\{\pm\delta_i\}$.
The Dynkin diagram is

  \begin{center}
\begin{picture}(120,30)
\put(0,10){\circle{3}}
\put(1.5,10){\line(1,0){27}}
\put(32,10){\line(1,0){10}}
\put(68,10){\line(1,0){10}}
\put(30,10){\circle{3}}
\put(80,10){\circle{3}}
\put(47,10){$\ldots$}

\put(110,10){\circle*{3}}

\put(81,8.5){\line(1,0){24.5}}
\put(81,11.5){\line(1,0){24.5}}
\put(101.5,7){$>$}

\put(107,14){\smaller[2]$\delta_\m$}

\put(0,14){\smaller[2]$\delta_1-\delta_2$}
\put(20,0){\smaller[2]$\delta_2-\delta_3$}
\put(70,0){\smaller[2]$\delta_{\m-1}-\delta_{\m}$}
 \end{picture}
\end{center}
$\bullet\ \o\s\p(2\n|2\m)$. Here we also distinguish two cases: $\n\neq 1$ and $\n=1$.\\
For $\n\neq1$, we have
$$
 \begin{cases}
   \Rm^+=\{\ve_i\pm\ve_{j}\}_{i<j}\cup\{\delta_i\pm\delta_{j}\}_{i< j}\cup\{\delta_j \pm \ve_i,2\delta_i\},\\
  \Pi=\{\delta_i-\delta_{i+1}\}_{i=1}^{\m-1}\cup \{\delta_\m - \ve_1\} \cup \{\ve_i -\ve_{i+1}\}_{i=1}^{\n-1} \cup \{\ve_{\n-1} + \ve_\n\}.\\
 \end{cases}
 $$
with $\Rm_0=\{\pm \ve_i\pm \ve_j,\pm2\delta_i, \pm\delta_i \pm\delta_j,\}_{i\neq j}$,  $\Rm_1=\{\pm\delta_i\pm\ve_j\}$, and the Dynkin diagram

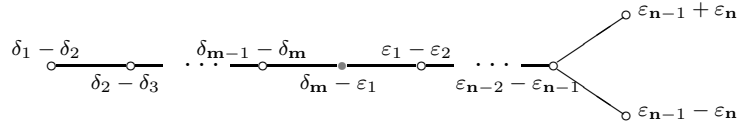
\begin{center}
\begin{picture}(200,30)
\put(0,10){\circle{3}}
\put(30,10){\circle{3}}

\put(80,10){\circle{3}}
\put(110,10){\color{gray}\circle*{3}}
\put(140,10){\circle{3}}

\put(82,10){\line(1,0){26}}
\put(112,10){\line(1,0){26}}

\put(142,10){\line(1,0){10}}
\put(178,10){\line(1,0){10}}
\put(160,10){$\ldots$}
\put(190,10){\circle{3}}

\put(2,10){\line(1,0){26}}
\put(32,10){\line(1,0){10}}
\put(68,10){\line(1,0){10}}
\put(50,10){$\ldots$}

\put(55,14){\smaller[2]$\delta_{\m-1}-\delta_{\m}$}

\put(-15,14){\smaller[2]$\delta_1-\delta_2$}
\put(15,1){\smaller[2]$\delta_2-\delta_3$}

\put(94,1){\smaller[2]$\delta_{\m}-\ve_1$}
\put(125,14){\smaller[2]$\ve_1-\ve_2$}
\put(153,1){\smaller[2]$\ve_{\n-2}-\ve_{\n-1}$}

\put(191.5,11.5){\line(3,2){25}}
\put(191.5,8.5){\line(3,-2){25}}
\put(217.5,29){\circle{3}}
\put(217.5,-9){\circle{3}}

\put(222,-9){\smaller[2]$\ve_{\n-1}-\ve_{\n}$}
\put(222,29){\smaller[2]$\ve_{\n-1}+\ve_{\n}$}
 \end{picture}
\end{center}
When  $\n=1$, the root system is
 $$\begin{cases}
   \Rm^+=\{\delta_i\pm\delta_{j}\}_{i< j}\cup\{\delta_j \pm\ve_1,2\delta_i\},\\
  \Pi=\{\delta_i-\delta_{i+1}\}_{i=1}^{\m-1}\cup \{\delta_\m - \ve_1\}\cup \{\delta_\m + \ve_1\},
 \end{cases}$$
with $\Rm_0=\{\pm2\delta_i, \pm\delta_i \pm\delta_j,\}_{i\neq j}$ and $\Rm_1=\{\pm\delta_i\pm\ve_1\}$.
Its Dynkin diagram is
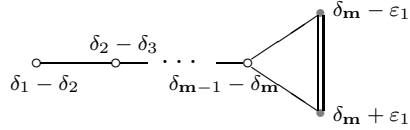
\begin{center}
\begin{picture}(100,60)

\put(0,30){\circle{3}}
\put(1.5,30){\line(1,0){27}}

\put(32,30){\line(1,0){10}}
\put(68,30){\line(1,0){10}}
\put(30,30){\circle{3}}
\put(47,30){$\ldots$}
\put(80,30){\circle{3}}
\put(81.5,31.5){\line(3,2){25}}
\put(81.5,28.5){\line(3,-2){25}}
\put(107.5,49){\color{gray}\circle*{3}}
\put(107.5,11){\color{gray}\circle*{3}}

\put(-10,20){\smaller[2]$\delta_1-\delta_2$}
\put(20,35){\smaller[2]$\delta_2-\delta_3$}
\put(50,20){\smaller[2]$\delta_{\m-1}-\delta_\m$}
\put(112,8){\smaller[2]$\delta_\m+\ve_1$}
\put(112,48){\smaller[2]$\delta_\m-\ve_1$}

\put(106.5,12){\line(0,1){36}}
\put(108.5,12){\line(0,1){36}}

\put(170,27)

\end{picture}
\end{center}
$\bullet\ \s\p\o(2\n|2\m)$ has the root system
$$
  \begin{cases}
   \Rm^+=\{\ve_i\pm\ve_{j}\}_{i<j}\cup\{\delta_i\pm\delta_{j}\}_{i< j}\cup\{\delta_j \pm \ve_i,2\ve_i\},\\
  \Pi=\{\delta_i-\delta_{i+1}\}_{i=1}^{\m-1}\cup \{\delta_\m - \ve_1\} \cup \{\ve_i -\ve_{i+1}\}_{i=1}^{\n-1} \cup \{2\ve_\n\},\\
 \end{cases}
$$
with
$\Rm_0=\{\pm \ve_i\pm \ve_j,\pm2\ve_i, \pm\delta_i \pm\delta_j,\}_{i\neq j}$, $\Rm_1=\{\pm\delta_i\pm\ve_j\}$, and
the Dynkin diagram
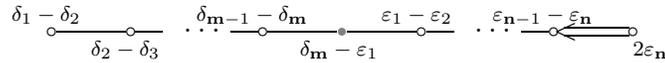
\begin{center}
\begin{picture}(200,30)
\put(0,10){\circle{3}}
\put(30,10){\circle{3}}

\put(80,10){\circle{3}}
\put(94,1){\smaller[2]$\delta_{\m}-\ve_1$}
\put(110,10){\color{gray}\circle*{3}}
\put(140,10){\circle{3}}

\put(82,10){\line(1,0){26}}
\put(112,10){\line(1,0){26}}

\put(142,10){\line(1,0){10}}
\put(178,10){\line(1,0){10}}
\put(160,10){$\ldots$}
\put(190,10){\circle{3}}

\put(2,10){\line(1,0){26}}
\put(32,10){\line(1,0){10}}
\put(68,10){\line(1,0){10}}
\put(50,10){$\ldots$}

\put(55,14){\smaller[2]$\delta_{\m-1}-\delta_{\m}$}

\put(-15,14){\smaller[2]$\delta_1-\delta_2$}
\put(15,1){\smaller[2]$\delta_2-\delta_3$}

\put(125,14){\smaller[2]$\ve_1-\ve_2$}
\put(167,14){\smaller[2]$\ve_{\n-1}-\ve_{\n}$}

\put(220,1) {\smaller[2]$2\ve_{\n}$}

\put(220,10){\circle{3}}

\put(194,8.5){\line(1,0){24.5}}
\put(194,11.5){\line(1,0){24.5}}
\put(190,7){$<$}
 \end{picture}
\end{center}

 \subsection{Natural Representations}
In this section we describe an irreducible representation of $U_q(\g)$ on the graded vector space
$\End(\C^{N|2\m})$. Let $\{v_i\}\subset \C^{N|2\m}$ be the standard homogeneous weight basis.
We call the grading symmetric (see Definition \ref{Weyl-operator} below) if $|v_i|=|v_{i'}|$ for all $i$, where
$$
i'=N+2\m+1-i.
$$
We  choose a symmetric grading  with minimal number of odd simple roots:
$$
\left(\underbrace{1, \ldots, 1}_\m ;\underbrace{0,\ldots,0}_{N};\underbrace{1, \ldots, 1}_\m\right).
$$
We will use notation $|i|=|v_i|$ for all basis vectors $v_i$.

Let  $\{e_{ij}\}_{i,j=1}^{N+2m}\subset  \End(\C^{N|2\m})$  be the standard matrix basis with multiplication  $e_{ij}e_{mn}=\dt_{jm}e_{in}$. An irreducible representation $\pi:U_q(\g) \rightarrow \End(\C^{N|2\m})$ is defined by the following assignment:
$$
q^{h_{\zeta_i}} \mapsto\ \sum_{j=1}^{N}q^{ (-1)^{|i|}\delta_j^i} e_{jj},\quad\mathrm{if}\ i\leq n+1,\   \g=\g\l(n+1),
$$
\be
q^{h_i} \mapsto&&
\hspace{-15pt}
\begin{cases}
\sum_{j=1}^{N+2\m}q^{(-1)^{|i|}( \delta_j^i-\delta_j^{i'})+(-1)^{|i+1|}(- \delta_j^{(i+1)}+\delta_j^{(i+1)'})}e_{jj},&
\mathrm{if}\ i< \n+\m,\
\g=\begin{cases}\o\s\p(2\n+1|2\m), \\
\o\s\p(2\n|2\m),\\  \s\p\o(2\n|2\m),\end{cases}
\\
\sum_{j=1}^{N+2\m}q^{(-1)^{|i|}(\delta_j^i-\delta_j^{i'})}e_{jj},&\mathrm{if}\ i=\m+\n ,\ \g=\o\s\p(2\n+1|2\m), \\
\sum_{j=1}^{N+2\m}q^{2\delta_j^{i}-2\delta_j^{i'}}e_{jj},&\mathrm{if}\ i=\m+\n ,\ \g=\s\p\o(2\n|2\m), \\
\sum_{j=1}^{N+2\m}q^{(-1)^{|i-1|}(\delta_j^{(i-1)}-\delta_j^{(i-1)'})} q^{(-1)^{|i|}( \delta_j^{i}-\delta_j^{i'})}e_{jj},&\mathrm{if}\ i=\m+\n ,\ \g=\o\s\p(2\n|2\m), \\
     \end{cases}
\nonumber
\ee
\be
e_i\mapsto&&\hspace{-15pt}
\begin{cases}
e_{i,i+1},&\mathrm{if}\ i\leq N+2\m-1  ,\ \g=\g\l(N|2\m),\\
q^{-\delta_i^\m}e_{i,i+1}-(-1)^{(|i|)(|i+1|+1)}e_{(i+1)',i'},& \mathrm{if}\ i\leq \m+\n, \
\g=\begin{cases}\o\s\p(2\n+1|2\m), \\
\o\s\p(2\n|2\m),\\  \s\p\o(2\n|2\m),\end{cases}\\
q^{-\delta_{i-1}^\m}e_{i-1,i+1}-(-1)^{|i-1|}e_{(i+1)',(i-1)'},&\mathrm{if}\ i=\m+\n ,\ \g=\o\s\p(2\n|2\m), \\
e_{i,i+1},&\mathrm{if}\ i=\m+\n  ,\ \g=\s\p\o(2\n|2\m),\\
     \end{cases}
     \nn
\ee
\begin{equation*}
f_{i}\mapsto\begin{cases}
  -e_{i+1,i}, & \mbox{if }\ i=\m,\ \g=\g\l(N|2\m), \\
  -q e_{i+1,i}+e_{i',(i+1)'}, & \mbox{if }\ i=\m,\ \g=\o\s\p(N|2\m), \  \mathrm{or}\ \g=\s\p\o(2\n|2\m),\\
   -q e_{i+1,i-1}+e_{(i-1)',(i+1)'}, & \mbox{if }\ i=\m+1,\ \g=\o\s\p(2|2\m),
\end{cases}
\end{equation*}
and  the similar assignment  to the remaining $f_i$ as for $e_i$ with  $e_{ij}$  changed  to $e_{ji}$.
In the tensor square of this representation,  the universal R-matrix mentioned in
Section 2.1 has the following expression, up to a scalar multiple:
\begin{equation}\label{R-sl}
R=\sum_{i,j}q^{(-1)^{|i|}\dt_{ij}} e_{ii}\tp e_{jj}+\omega\sum_{j<i}(-1)^{|j|}e_{ij}\tp e_{ji},
\end{equation}
for  general  linear   $\g$ \cite{Is,Zh} and
\begin{equation}\label{R-osp}
R=\sum_{i,j=1} q^{(-1)^{|j|}(\delta_{ij}-\delta_{ij'})}e_{ii}\otimes e_{jj} +\omega\sum\limits_{\substack{j,i=1\\ j<i}} ((-1)^{|j|}e_{ij}\otimes e_{ji} -(-1)^{|i|+|j|+|i||j|}\kappa_i\kappa_jq^{\rho_i-\rho_j} e_{ij}\otimes e_{i'j'}),
\end{equation}
for  ortho-symplectic $\g$ \cite{Is,DGL}.
Here
{\footnotesize$$
  (\rho_i)=\begin{cases}
    (k-\m,\dots,k-1;k-1,\dots,\frac{1}{2},0,-\frac{1}{2},\dots,1-k;1-k,\dots,\m-k),\ k=\frac{2\n+1}{2}, \ \mathrm{for}\  \o\s\p(2\n+1|2\m),\\
    (\n-\m,\dots,\n-1;\n-1,\dots,1,0,0,-1,\dots,1-\n;1-\n,\dots,\m-\n),\ \mathrm{for}\  \o\s\p(2\n|2\m), \\
    (\n-\m+1,\dots,\n;\n,\dots,1,-1,\dots,1-\n;1-\n,\dots,\m-\n-1), \ \mathrm{for}\  \s\p\o(2\n|2\m),
  \end{cases}
 $$}
 $$(\kappa_i)= \begin{cases}
    (-1,\dots,-1;1,\dots,1;1,\dots,1),\ \mathrm{for}\  \o\s\p(N|2\m),\\
    (-1,\dots,-1;1,\dots,1,-1,\dots,-1;-1,\dots,-1), \ \mathrm{for}\  \s\p\o(2\n|2\m).
  \end{cases}
$$
Let us emphasise that the tensor product in (\ref{R-sl}) and  (\ref{R-osp}) is graded. The identity  $e_{ij}\tp e_{lk}=(-1)^{(i+j)(l+k)} e_{lk}\tp e_{ij}$
is satisfied for all matrix units.


\section{Graded Reflection Equation}
\label{SecKmat}
The non-graded version of RE appeared in mathematical physics literature \cite{KSkl,KSS} and triggered
the theory of quantum symmetric pairs in \cite{NS,NDS,Let}. The graded RE has been of interest as well, mostly in
the spectral parameter dependent form \cite{AACLFR,AACLFR1,L,DK}.

In this section, we present a class of solutions to the constant RE for the general linear and orthosymplectic quantum supergroups.
We are interested in invertible even RE-matrices.  Solutions for the general linear supergroups are taken  from our recent paper \cite{AlgMS1}.

\subsection{$\Z_2$-graded Reflection Equation}\label{sec-z2-re}

Suppose that $V$ is a graded vector space and $\Ac=\End(V)$ is the corresponding graded matrix algebra.
An invertible element $R\in \Ac\tp \Ac$ is called an $R$-matrix if it satisfies  Yang-Baxter equation
$$
R_{12}R_{13}R_{23}=R_{23}R_{13}R_{12},
$$
where the subscripts indicate the tensor factor in the graded tensor cube of $\End(V)$.
An even element $P=\sum_{i=1}^{n}(-1)^{|j|}e_{ij}\tp e_{ji}\in \End(V)\tp \End(V)$ is called graded permutation. It flips the
tensor factors in $V\tp V$ by the rule
$$
P(v\tp w)=(-1)^{|v||w|}w\tp v
$$
for all homogeneous $v,w\in V$.
The operator $S=PR\in \End(V)\tp \End(V)$ satisfies the braid relation
$$
S_{12}S_{23}S_{12}=S_{23}S_{12}S_{23}.
$$
A matrix $K\in \End(V)$ is said to satisfy  RE if the identity
\be
\label{REuntw}
SK_2S K_2= K_2 S K_2 S
\ee
holds true in $\End(V)\tp \End(V)$. In particular, a scalar matrix satisfies this equation. This solution
is not interesting and should be considered as trivial.
We will  assume that $K$ is even. Note that the particular form (\ref{REuntw}) is related to left coideal subalgebras, cf. Section \ref{SecCoidSubA}.

 K-matrices amount to cylindric braiding in  a suitable category of $U_q(\g)$-modules, \cite{DH}, via a sort of fusion procedure similar to
R-matrices.
Cylindric brading is a collection of intertwiners $K_M\in \End_{U_q(\k)}(M)$ parameterized by  a $U_q(\g)$-module $M$ that satisfy
Reflection Equation
$$
S_{M,M}(\id \tp K_M)S_{M,M}(\id \tp K_M)=(\id \tp K_M)S_{M,M}(\id \tp K_M)S_{M,M},
$$
where $S_{M,M}$ is the product $P_{M,M}R_{M,M}$ of the permutation $P_{M,M}$ and the image $R_{M,M}$ of the universal
matrix $\Ru$  in $\End(M\tp M)$.
In the non-graded case, $K_M$ can be  obtained as the representation of an element  $\Kc$ (universal K-matrix) from a completion of
$U_q(\g)$ that satisfies
$$
\Delta(\Kc)=\Ru_{21}\Kc_1\Ru_{12} \Kc_2.
$$
The universal K-matrix is constructed in  \cite{BK} for symmetric pairs $(\g,\k)$ with symmetrizable Kac-Moody algebra $\g$ of finite type,
in development of \cite{BW} for the general linear case.
Its construction is extended for pseudo-symmetric pairs in \cite{RV} and to all symmetrizable Kac-Moody algebras in \cite{AV}.
It is natural to expect a  generalization of the universal K-matrix for  quantum super-spherical pairs.

Equation (\ref{REuntw}) admits generalizations with applications to integrable systems \cite{FM}.
We present here two  versions which are of relevance to left coideal subalgebras.
Suppose that $\si\colon U_q(\g)\to U_q(\g)$  is a  superinvolutive superalgebra graded anti-automorphism and supercoalgebra automorphism (Chevalley superinvolution) such that
$\pi\circ \si(x)=\pi^t(x)$ for all $x\in U_q(\g)$, where $\pi$ is the representation of $U_q(\g)$ on $V$ and $t$ is the matrix super-transposition. If $R^{t_1t_2}=R_{21}$, then the identity
\be
\label{REtw}
R_{21}K_1 R_{21}^{t_1}K_2=K_2R_{12}^{t_2}K_1R_{21}
\ee
is called twisted RE. Here $t_i$, $i=1,2$ designate the super-transposition  applied to the $i$-th tensor factor.

Now suppose that $\vartheta\colon U_q(\g)\to U_q(\g)$ is an involutive  Hopf superalgebra automorphism and $\theta\colon \End(V)\to \End(V)$ a matrix
super-algebra automorphism (conjugation with a fixed invertible even matrix) such that $\pi\circ \vartheta(x)=\theta\circ\pi(x)$ for all $x\in U_q(\g)$. Suppose also that
$R^{\theta_1\theta_2}=(\theta\tp \theta)(R)=R$.
Then one can consider a twisted RE in  the form
\be
\label{REtw1}
R_{21}K_1R^{\theta_1}_{12}K_2=K_2R_{21}^{\theta_2}K_1R_{12},
\ee
In the special case $\vartheta =\id$ we return to the equation (\ref{REuntw}).
We will focus on (\ref{REuntw}) in the current paper.

In what follows, we mean by  RE only the form (\ref{REuntw}) unless otherwise is explicitly stated.
The case of the general linear quantum supergroup has been studied in detail in \cite{AlgMS1}, where the full list of solutions to RE is given relative
to an arbitrary grading of $V$.
They turn out to be exactly the even matrices solving the non-graded RE. Invertible solutions
occur only for a (arbitrary) symmetric grading
(here we use a chance to correct a sloppy remark  on page 6 of \cite{AlgMS1})
and have the form
\be
\label{A-gl}
 K=(\la+\mu)\sum_{i=1}^{m}e_{ii}+\la\sum_{i=m+1}^{N+2\m-m} e_{ii}+\sum_{i=1}^{m}y_{i}e_{ii'}+\sum_{i=1}^{m}y_{i'}e_{i'i},
\ee
where $y_i$ are complex numbers subject to  $y_iy_{i'}=-\la\mu\not =0$.
This matrix satisfies the RE with the R-matrix (\ref{R-sl}).
We present a class of solutions for the ortho-symplectic quantum supergroups in the next section.

\subsection{K-matrices for ortho-symplectic quantum groups}

In this section, we find even K-matrices for the ortho-symplectic quantum  supergroups.
We introduce a grading on the underlying vector space $V = \C^{N|2\m}$ by setting
$$
|i| =
\left \{
\begin{array}{cccc}
  1, & i\leqslant \m &\mbox{or} & N+\m<i,\\
  0, & \m<i\leqslant \m+N.
\end{array}
\right.
$$
This is a symmetric grading with the least number of odd simple roots.
We are looking for solutions of  the RE in the following three forms:
$$
A=\sum_{i=1}^{N+2\m}x_{i}e_{ii}+\sum_{i=1}^{N+2\m}y_{i}e_{ii'},
\quad
B=\sum_{i=1}^{N+2\m}x_ie_{ii}+\sum_{i=1}^{N+2\m}y_i(e_{i, i'-1}-e_{i+1,i'}),
$$
$$
 C=\sum\limits_{\substack{i=1\\ i\neq n}}^{2n-1}x_i(e_{i, i'-1}-e_{i+1,i'})+x_n e_{n,n'}+x_{n'}e_{n',n},\quad  \mathrm{for}\ \g=\o\s\p(2|4\m).
$$
The matrices $A$ and $B$ will depend on an integer parameter  $m$  subject to inequality $1\leq m\leq \m$ in the forthcoming theorem.
\begin{thm} \label{s-graded}
The following matrices satisfy the Reflection Equation associated with  ortho-symplectic quantum groups:
  \begin{equation}
  A=\sum_{i=1}^m \lambda(1-\kappa_m\kappa_{m'}q^{-2\rho_{m}}) e_{ii}+\sum_{m<i<m'}\lambda e_{ii}+\sum\limits_{\substack{i\leqslant m} }(y_ie_{i, i'}+y_{i'}e_{i', i'}),
\end{equation}
 \begin{equation}
\begin{split}
B=\sum_{i=1}^m& \lambda(1+\kappa_{m-1}\kappa_{m'+1}q^{-2(\rho_{m}+1)}) e_{ii}+\sum_{m<i<m'}\lambda e_{ii}\hspace{30pt}
\\ &
\hspace{30pt}+\sum\limits_{\substack{i\leqslant m
\\
i= 1\!\!\!\!\mod 2}}\bigl(z_i(e_{i, i'-1}-e_{i+1,i'})+z_{i'-1}(e_{i'-1, i}-e_{i',i+1})\bigr),
\end{split}
 \end{equation}
  \begin{equation}
 C=\sum\limits_{\substack{i=1\\i= 1\!\!\!\!\mod 2 \\ i\neq n}}^{2n-1}x_i(e_{i, i'-1}-e_{i+1,i'})+x_n e_{n,n'}+x_{n'}e_{n',n},
\end{equation}
where the parameters $y_i,z_i, x_i \in \C$ satisfy the conditions
\begin{itemize}
  \item $y_iy_{i'}=\kappa_m\kappa_{m'}\lambda^2 q^{-2\rho_{m}}$,
  \item $z_{i}z_{i'-1}=-\kappa_{m-1}\kappa_{m'+1}\lambda^2 q^{-2(\rho_{m}+1)}$,
\item $x_{i}x_{i'-1}=x_nx_{n'}$.
\end{itemize}
\end{thm}
\begin{proof}
 Direct calculation.
\end{proof}
The presented K-matrices exhaust all of invertible solutions to the RE for
general linear $\g$, and they all are necessarily even \cite{AlgMS1}.
A full classification of K-matrices for orthogonal and symplectic
$\g$ is unknown even in the non-graded case.  Theorem \ref{s-graded}
gives the "simplest" examples, see the discussion in Section \ref{Sec-Gr-Satake}.
They  have the following shape:
\begin{center}
\begin{picture}(110,110)
\put(-30,45){$A=$}
\put(-5,-5){\line(0,1){110}}
\put(110,-5){\line(0,1){110}}
\put(75,30){\line(0,1){40}}
\put(35,30){\line(0,1){40}}
\multiput(15,95)(3,-3){3}{\circle*{1}}
\multiput(52,52)(3,-3){3}{\circle*{1}}
\multiput(95,95)(-3,-3){3}{\circle*{1}}
\multiput(8,8)(3,3){3}{\circle*{1}}

\put(100,100){$\scriptstyle y_{1}$}

\put(0,100){$\scriptstyle \lambda+\mu$}
\put(20,80){$\scriptstyle \lambda+\mu$}

\put(37,62){$\scriptstyle \la$}

\put(68,32){$\scriptstyle \la$}

\put(80,80){$\scriptstyle y_{m}$}

\put(20,20){$\scriptstyle y_{m'}$}
\put(0,0){$\scriptstyle y_{1'}$}

\end{picture}\>\>,
\quad\quad\quad\quad\quad\quad
\begin{picture}(110,110)
\put(-30,45){$B=$}
\put(-5,-5){\line(0,1){110}}
\put(110,-5){\line(0,1){110}}
\put(75,30){\line(0,1){40}}
\put(35,30){\line(0,1){40}}
\multiput(15,95)(3,-3){3}{\circle*{1}}
\multiput(52,52)(3,-3){3}{\circle*{1}}
\multiput(95,95)(-3,-3){3}{\circle*{1}}
\multiput(10,8)(3,3){3}{\circle*{1}}

\put(95,100){$\scriptstyle z_{1}\nu$}

\put(0,100){$\scriptstyle \lambda+\mu$}
\put(20,80){$\scriptstyle \lambda+\mu$}

\put(37,62){$\scriptstyle \la$}

\put(68,32){$\scriptstyle \la$}

\put(75,80){$\scriptstyle z_{m-1}\nu$}

\put(20,20){$\scriptstyle z_{m'}\nu$}
\put(0,0){$\scriptstyle z_{2'}\nu$}
\end{picture}
\>\>,
\end{center}
\begin{center}
\begin{picture}(90,90)
\put(-30,45){$C=$}
\put(-5,-5){\line(0,1){90}}
\put(90,-5){\line(0,1){90}}

\multiput(75,75)(-3,-3){3}{\circle*{1}}
\multiput(10,8)(3,3){3}{\circle*{1}}

\put(48,48){$\scriptstyle x_{n}$}
\put(35,35){$\scriptstyle x_{n+1}$}

\put(75,80){$\scriptstyle x_{1}\nu$}

\put(55,60){$\scriptstyle x_{n-2}\nu$}

\put(20,20){$\scriptstyle x_{n+2}\nu$}
\put(0,0){$\scriptstyle x_{2'}\nu$}
\end{picture}
\>\>,
\end{center}
with  $\nu=\begin{bmatrix}
1&0\\
0&-1
\end{bmatrix}
$
 and $\mu= -\la\kappa_m\kappa_{m'}q^{-2\rho_{m}}$ for $A$, $\mu=\la \kappa_{m-1}\kappa_{m'+1}q^{-2(\rho_{m}+1)}$ for  $B$.
In both cases of  $A$ and $B$, the  blocks with diagonal entries of  $\la+\mu$ and $\la$  are, respectively, of size $m$ and  $N + 2\m - 2m$,
 with  $m\leq\m$ as stipulated.
We relate them with coideal subalgebras in $U_q(\g)$ in Section \ref{SecCoidSubA}.

\section{Classical super-spherical pairs}
In this section we define  Lie superalgebras $\k\subset \g$ that give rise to
coideal subalgebras in generalization of the Letzter theory for non-graded quantum groups.
\subsection{Weyl operator and spherical data }
\label{SecWOp}
 Let $\g$ be a Lie super-algebra that features a triangular decomposition with Cartan subalgebra $\h$,  and let  $\b\subset \g$ be one its Borel subalgebras
 containing $\h$.
\begin{definition}
A Lie super-algebra $\k\subset \g$ is called spherical if $\g=\k+\b$.
Then the pair $(\g,\k)$ is called spherical.
\end{definition}
It is known that, depending on the polarization and the corresponding choice of simple root basis,
 Borel subalgebras in $\g$ are generally  not isomorphic.
Thus, contrary to the non-graded case, this definition of sphericity depends
upon a choice of $\b$.
From now on we restrict our consideration to the case when $\g$ is either general linear or ortho-symplectic.
The choice of $\b$ is determined by a grading of the underlying natural module.

Let us fix a graded  basis of weights $\zt_i$, $i=1,\ldots, N$, of the natural $\g$-module $\C^N$. They generate
the weight lattice $\La$ of $\g$.
As before we use the notation $ i'=N+1-i$ for all $i=1,\ldots, N$.
An element $w\in SL(\La)$ is said to be even if $w(\Rm_i)=\Rm_i$, $i=0,1$.
A grading on $\C^N$ splits $\{\zt_i\}_{i=1}^N$  to even and odd subsets. An element  $w\in SL(\La)$
is clearly even if it  preserves this decomposition.

\begin{definition}
\label{Weyl-operator}
\begin{enumerate}
  \item A unique $\Z$-linear map $w_\g\colon \La\mapsto \La$ defined by the assignment
$\zt_i\mapsto \zt_{i'}$, $i=1,\ldots, N,$ is called
Weyl operator.
   \item
  The grading on $\C^N$ is called symmetric if the Weyl operator is even.
\end{enumerate}
\end{definition}
Clearly, a  grading is  symmetric  if and only if  every inversion
$\si_i\colon \zt_i\mapsto \zt_{i'}$, $ \zt_j \mapsto \zt_j$, $j\not = i,i'$, which extends to an element of $SL(\La)$, is even.
It is also obvious that  even $w_\g=\prod_{i=1}^{N}\si_i$ extends to an  involutive  orthogonal operator $w_\g\colon \h^*\to \h^*$.

From now on we will work  only with the minimal symmetric grading on $\g$.
Consider the Dynkin diagram $D$ of a basic Lie superalgebra $\g$. If we discard the grading information, then we  get
a diagram $\tilde D$ that we call shape of $D$. In the special case of  $\g=\o\s\p(2|2\m)$,  we understand by $\tilde D$ the Dynkin diagram of $\s\o(2+2\m)$.
Let $\tilde W$ denote the group of  automorphisms of  $\tilde D$. Denote by $\tilde W_0$ its subgroup that preserves the grading.
\begin{lemma}
\label{transposition}
The group $\tilde W_0$  preserves the root system $\Rm$ and the weight lattice
 $\La$.
\end{lemma}
\begin{proof}
For general linear $\g$, the group $\tilde W$ is generated by all permutations of $\zt_i$. For ortho-symplectic  $\g$,
it is generated by transpositions $\zt_i\leftrightarrow \zt_j$ and
inversions $\si_i$ taking $\zt_i$ to $\zt_{i'}=-\zt_{i}$, $i,j\leqslant \frac{N}{2}$.
Now the statement follows from the explicit description of the root systems given in Section \ref{SecBasQSG}.
\end{proof}
\noindent
\begin{remark}
\label{Rem-Weyl-op}
\em
Thus defined $w_\g$ is analogous to the longest element of the Weyl group of the ordinary simple Lie algebras.
In the case of $\g=\s\o(2n)$ of odd $n$, the operator $w_\g$ differs from the longest element by the non-trivial automorphism of the Dynkin diagram
(the tail flip). This operator can be equally used for construction of  non-graded pseudo-symmetric pairs, because it flips the highest and lowest weights of
the minimal $\g$-module, cf. Lemma \ref{high-low}. It is easier to  define than the "honest" analog of the longest Weyl group element
for $\g=\s\p\o(2\n|2\m)$ with odd $\m$. This explains our choice of $w_\g$.
\end{remark}

Pick a subset $\Pi_\l\subset \Pi$, then put  $\bar \Pi_\l=\Pi \backslash \Pi_\l$, and generate a  subalgebra $\l=\langle e_\al,f_\al\rangle_{\al\in \Pi_\l}\subset \g$.
It is a direct sum of subalgebras, $\l=\sum_{i}\l_i$, corresponding to connected components of $\Pi_\l$.
If $\l_i$ is of type A, then set $\hat \l_i\subset \g$ to be the natural $\g\l$-extension of $\l_i$ and leave $\hat \l_i=\l_i$ otherwise.
Denote by $\hat \l=\op_i \hat\l_i\subset \g$ and by $\h_{\hat\l}^*$ its Cartan subalgebra.
The restriction of the canonical inner product from $\h^*$ to $\h^*_{\hat \l}$ is non-degenerate.

For the $i$-th connected component of $\Pi_\l$ denote $w_{\l_i}=w_{\hat \l_i}$ and define $w_\l=\prod_i w_{\l_i}\in \End(\h^*_\l$),
the Weyl operator  of the subalgebra $\hat \l$.

\begin{definition}
\label{admissible}
  We call $\Pi_\l$ admissible if the grading on $\g$ induces a minimal symmetric grading on each connected component of $\hat \l$.
\end{definition}
\noindent
Here is an easy consequence of this definition.
\begin{propn}
\label{odd-Levi-iso}
  An admissible set $\Pi_\l$ cannot contain an isolated grey odd root.
\end{propn}
\begin{proof}
  Indeed, if $\al\in \Pi_\l$ is such a root, then the corresponding subalgebra  $\hat \l_k$ is isomorphic to  $\g\l(1|1)$. It has no symmetric grading, and the operator $w_{\l_k}$ is not even.
\end{proof}
\noindent
Note that an isolated black odd root is possible: one can take  $\g=\o\s\p(1|2\m)$ and $\Pi_\l=\{\al_\m\}$.

If $\Pi_\l$ is admissible, then the  operator  $w_\l$ is even.
    \begin{lemma}
  For admissible $\Pi_\l$, the operator $w_\l\in \End(\h^*_{\hat \l})$ extends to an involutive operator on $\h^*$ that is even on the weight lattices  of $\l$ and $\g$.
\end{lemma}
\begin{proof}
The set of weights $\La_\g(V)=\{\zt_i\}$ of the basic module splits into a disjoint union $\La_{\hat \l}(V)\cup \overline{\La_{\hat \l}(V)}$.
The operator  $w_\l$ preserves the weight lattice of $\hat\l$ and its root system by construction.
It acts as a permutation on $\La_\g(V)$ identical on $\overline{\La_{\hat \l}(V)}$, therefore it preserves the weight lattice of $\g$.
Direct  examination shows $w_\l(\Rm)=\Rm$ for the root system of $\g$. Furthermore, it preserves the parity of all weights from  $\La(V)$, therefore it is even.
Since the inner product on $\h^*_{\hat \l}$ is non-degenerate, we have orthogonal decomposition $\h^*=(\h_{\hat\l}^*)^\perp\op \h^*_{\hat\l}$,
where $\h^*_{\hat\l}$ is spanned by $\La_{\hat \l}(V)$ whereas  $(\h_{\hat\l}^*)^\perp$ by $\overline{\La_{\hat \l}(V)}$.
Finally, $w_\l$ extends as the identity operator on $(\h_{\hat\l}^*)^\perp$, therefore it is orthogonal and involutive.
\end{proof}

The vector subspaces
$$\mathfrak{m}_+=\sum_{\al\in \Rm^+_{\g}\>- \>\Rm^+_{\l}} \g_\al,
\quad
\mathfrak{m}_-=\sum_{\al\in \Rm^+_{\g}\>- \>\Rm^+_{\l}} \g_{-\al},
$$
are graded $\l$-modules.
For $\al\in \bar \Pi_\l$ let $V^\pm_\al\subset \g_\pm$ denote the $\l$-submodule generated by $e_{\pm \al}\in \g_{\pm \al}$.
\begin{lemma}
\label{high-low}
Suppose that $\g$ is a basic matrix Lie superalgebra and $\Pi_\l\subset \Pi$ is admissible.
 Then the operator $w_\l$ flips  the highest and lowest weights of $V^\pm_\al$ for each $\al \in \bar \Pi_\l$.
\end{lemma}
\begin{proof}
The subalgebra $\l$ is a sum of $\l=\sum_{i} \l_i$ of basic Lie superalgebras $\l_i$ corresponding to
connected components of the Dynkin diagram of $\l$.
In all cases excepting $\g=\o\s\p(2\n|2\m)$, the modules $V^\pm_\al$ are tensor products of smallest fundamental $\l_i$-modules,
for which $w_\l$ does the job.

We are left to consider the situation of $\g=\o\s\p(2\n|2\m)$.
The module $V^\pm_\al$ which is not minimal for $\l_i\subset \l$  occurs in the following   two cases.
If $\Pi_{\l_i}=\{\al_{\n-2},\al_{\n-1},\al_{\n}\}$, then the non-trivial $\l_i$-submodule in $\mathfrak{m}_\pm $  is $\C^6$, for which the statement is obvious.
Another possibility is when $\al_{\n-1}\in \Pi_{\l_i}$, $\al_{\n}\in \bar \Pi_\l$
(or the other way around) and $\rk \>\l_i\geqslant 3$.
Now notice that $w_{\l_i}$ is even if and only if the tail roots and $\l_i$ are even (mind that we work with the minimal symmetric grading on $\g$), in which case $w_{\l_i}$
is  just the longest element
of the Weyl group of $\l_i$. Then the statement is true either.
 \end{proof}

From now on we will consider only admissible $\Pi_\l$. Suppose that  $\tau \in \Aut(\Pi)$  is an  even permutation that coincides with $-w_\l$ on $\Pi_\l$.
Set $\tilde \al =w_\l\circ \tau(\al)\in \Rm^+$ for all $\al\in \bar \Pi_\l$.
\begin{definition}
\label{triple}
  The triple $(\g,\l, \tau)$ is called pseudo-symmetric if
\be
(\mu+\tilde\mu,\al)&=0,& \quad \forall \mu\in \bar \Pi_\l, \quad \al\in \Pi_{\l},
\label{1nd-cond}\\
(\mu+\tilde\mu,\nu-\tilde \nu)&=0,& \quad \forall \mu,\nu\in \bar \Pi_\l.
\label{2nd-cond}
\ee
\end{definition}
Condition  (\ref{1nd-cond}) has the following algebraic interpretation.
\begin{lemma}
\label{tau-commutes w}
As a $\Z$-linear endomorphism of the root lattice,  $\tau$ commutes with $w_\l$.
\end{lemma}
\begin{proof}
By construction, $\tau$ and $w_\l$ commute when restricted to $\h^*_\l$.
It suffices to check that for simple roots from $\bar \Pi_\l$.

By Lemma \ref{high-low},  $w_\l$ takes the highest weight of an irreducible $\l$-module $V^\pm_\mu$, $\mu\in \bar \Pi_\l$, to the lowest weight and vice versa.
For each $\mu\in \bar \Pi_\l$ we have $w_\l(\mu)=\mu+\eta$ for some weight $\eta\in \Z_+\Pi_\l$ that satisfies
 $w_\l(\eta)=-\eta$ because $w_\l^2=\id$. The weight $\eta$ depends only on the projection of $\tau(\mu)$ to $\h^*_\l$,
therefore
$- \tau(\mu)=w_\l(-\tilde \mu)=-w_\l\bigl(\tau(\mu)\bigr)+\eta$ or
$\tau(\mu)+\eta=w_\l\bigl(\tau(\mu)\bigr)$
due to the condition (\ref{1nd-cond}).
Then, since $\tau(\eta)=-w_\l(\eta)$ for all $\eta\in \h^*_\l$,
$$
\tau\bigl(w_\l(\mu)\bigr)=\tau(\mu+\eta)=\tau(\mu)+\tau(\eta)=\tau(\mu)-w_\l(\eta)=\tau(\mu)+\eta=w_\l\bigl(\tau(\mu)\bigr),
$$
as required.
\end{proof}
Define a  linear map $\theta=-w_\l\circ \tau\colon \h^*\to \h^*$. It preserves $\Rm$ because so do $\tau$ and $w_\l$. Furthermore,
$\theta(\al)=-\tilde \al$ for $\al \in \bar \Pi_\l$
and $\theta(\al)=\al$ for $\al \in \Pi_\l$. As a product of two even operators on $\La$, it is even. The system of equalities  (\ref{1nd-cond}) and (\ref{2nd-cond})
is  equivalent to
\be
\label{gen-cond}
\bigl(\al+\theta(\al),\bt-\theta(\bt)\bigr)=0,\quad \forall \al,\bt\in \Pi.
\ee

\begin{propn}
\label{tau-invol-orth}
Condition (\ref{gen-cond})
is  fulfilled if and only if the permutation $\tau$ is  even and extends to an involutive orthogonal operator on $\h^*$
coinciding with $-w_\l$ on $\Pi_\l$.
\end{propn}
\begin{proof}
First of all remark that a linear operator being orthogonal and involutive is the same as symmetric and involutive,
or orthogonal and symmetric simultaneously. Also, an automorphism of the Dynkin diagram is exactly an orthogonal
operator on $\h^*$ that permutes the basis $\Pi$.

Condition (\ref{gen-cond}) is bilinear and therefore holds true for any pair of vectors from $\h^*$.
Setting $\al=\bt$ in  (\ref{gen-cond}) we find that $\theta$ is an isometry.
Then
 (\ref{gen-cond})
 translates to
$$
\bigl(\theta(\al), \bt\bigr)=\bigl(\al,\theta(\bt)\bigr), \quad \forall \al, \bt\in \Pi.
$$
It means that  $\theta$ is a symmetric operator.
Therefore it is an involutive orthogonal operator.
So is $w_\l$, that commutes with $\theta$ by Lemma \ref{tau-commutes w}. Hence $\tau$ is an involutive orthogonal
operator on $\h^*$.

Conversely, suppose that  $\tau$ is orthogonal, involutive, and coincides with $-w_\l$ on $\Pi_\l$.
Then
$$
(\al,\mu)=\bigl(\tau(\al),\tau(\mu)\bigr)=-\bigl(w_\l(\al),\tau(\mu)\bigr)=-\bigl(\al,w_\l\circ\tau(\mu)\bigr)=-(\al,\tilde \mu)
$$
 for all $\al\in \Pi_\l$ and $\mu\in \bar \Pi_\l$.
Thus, the condition (\ref{1nd-cond}) is fulfilled, and $\tau$ commutes with $w_\l$ by Lemma \ref{tau-commutes w}.
Therefore $\theta=-w_\l\circ \tau$ is orthogonal and involutive, which implies (\ref{gen-cond}).
\end{proof}
\noindent
We arrive at a necessary condition for a triple $(\g,\l, \tau)$ to be pseudo-symmetric.
\begin{corollary}
Suppose that $\tau$ satisfies conditions of Proposition \ref{tau-invol-orth}.
Then  $\l$-modules $\mathfrak{m}_+$ and $\mathfrak{m}_-$ are isomorphic.
\end{corollary}
\begin{proof}
    The Lie superalgebra $\l$ is a direct sum $\l=\sum_{i}\l_i$ of Lie superalgebras corresponding to the connected components of the Dynkin diagram of $\l$.
Each of $\l_i$ belongs to  a basic matrix  type. The modules $\mathfrak{m}_\pm$ are completely reducible and generated as Lie superalgebras by the submodules $V^\pm_\al$ with $\al \in \bar \Pi_\l$. The modules $V^\pm_\al$ are tensor products of the minimal modules for $\l_i$, except for
the special case of $\g=\o\s\p(2\n|2\m)$ considered in the proof of Lemma \ref{high-low}.
Therefore they are completely determined by inner products of their highest/lowest weights with $\Pi_\l$.
Now observe that  $V^\pm_\al\subset \g_\pm$ has an isomorphic partner $V^\mp_{\tau(\al)}\subset \g_\mp$ for each $\al\in \bar\Pi_\l$.
\end{proof}

The isomorphism between $V_\al^\pm $ and $V_{\al'}^\mp$,   $\al'=\tau(\al)$, is a consequence of (\ref{1nd-cond}).
We extend $\l$ by adding $\mathfrak{t}=\Span\{h_{\tilde \al}-h_\al\}$, for all $\al\in \bar \Pi_\l$, then  $V_\al^\pm $ and $V_{\al'}^\mp$
will be isomorphic as  $\l+\mathfrak{t}$-modules, provided  (\ref{2nd-cond}) is fulfilled. The same will be true for
$\mathfrak{m}_+$ and $\mathfrak{m}_-$.

Suppose that $\Pi_\l$ is admissible and $\tau$ satisfies the conditions of Proposition \ref{tau-invol-orth}.
Let $\c\subset \h$ denote the centralizer of $\l$ in $\h$.
For each $\al\in \bar \Pi_\l$ pick  $c_\al\in \C^\times $, $\grave c_{\al}\in \C$,  and $u_\al\in \c$ assuming $u_\al\not =0$ only if $\al$ is even,  orthogonal to $\Pi_\l$, and $\tilde  \al=\al=\tau(\al)$.
Put
\begin{equation}
\begin{array}{rcl}
y_\al &=& h_\al - h_{\tilde \al}, \\
x_\al &=& e_\al + c_\al f_{\tilde \al} + \grave c_{\al} u_\al,
\end{array}
\label{gen-spher-superpairs}
\end{equation}
for all $\al\in \bar \Pi_\l$.
Define a Lie subalgebra $\k\subset \g$ as the one generated by $\l+\t$ and by $x_\al$ with $\al\in  \bar \Pi_\l$.
\begin{definition}
  The pair of  Lie  superalgebras $\k\subset \g$ corresponding to a pseudo-symmetric triple $(\g,\l,\tau)$  is called pseudo-symmetric.
\end{definition}
\noindent
The  complex numbers $c_\al,\grave c_{\al}$ in (\ref{gen-spher-superpairs}) are called mixture parameters.

Next we prove that pseudo-symmetric pairs are spherical.
Denote by  $\p^0=\h +\l$ the Levi subalgebra with commutant $\l$ and set
$$
\p_\pm^1=\sum_{\mu\in \Pi_{\g/\l}}V_{\mu}^\pm +\p^0,
\quad
\p_\pm^{m+1}=[\p_\pm^m,\p^1_\pm], \quad m>0.
$$
We have  is an $\l$-equivariant filtration of parabolic subalgebras $\p_\pm$:
$$
\p^0\subset \p_\pm^1\subset \ldots \subset \p_\pm.
$$
\begin{lemma}
\label{m+n-}
  For all $m,n\geqslant 1$, there is an inclusion $[\p^m_+,\p^n_-]\subset \p^m_++\p^n_-$.
\end{lemma}
\begin{proof}
For $m=n=1$ this readily follows by root arguments. Suppose this is proved for $m=1$ and $n\geqslant 1$.
The (graded) Jacobi identity gives
$$
[\p^1_+,\p^{n+1}_-]= [\p^1_+,[\p^1_-,\p^{n}_-]]\subset [\p^1_++\p^1_-,\p^{n}_-]+[\p^1_-,\p^1_++\p^{n}_-]\subset \p^1_++\p^{n+1}_-,
$$
which proves the formula for $m=1$ and all $n$. Now suppose that it is true for $m\geqslant 1$.
Then
$$
[\p^{m+1}_+,\p^{n}_-]= [[\p^m_+,\p^1_+],\p^{n}_-]\subset  [\p^m_++\p^{n}_-,\p^1_+]+ [\p^m_+,\p^1_++\p^{n}_-]\subset \p^{m+1}_++\p^n_-,
$$
which completes the proof.
\end{proof}
\begin{corollary}
\label{cor_+_-}
  Let $m_+$ be the number of pluses and $m_-$ the number of minuses in a sequence $(\ve_1,\ldots, \ve_k)$,
  where $\ve_i\in \{\pm\}$.
Then
$$
[\ldots[\p^1_{\ve_1}, \p^1_{\ve_2}],\ldots, \p^1_{\ve_k}]\subset \p^{m_+}_++ \p^{m_-}_-.
$$
\end{corollary}
\begin{proof}
For $k=2$, this follows from the definition of $\p_\pm^m$ and from Lemma \ref{m+n-}. Suppose the statement is proved for $k\geqslant 2$. Then, by induction,
$$
[\ldots[\p^1_{\ve_1}, \p^1_{\ve_2}],\ldots, \p^1_{\ve_{k+1}}]\subset [\p^{m_+}_++ \p^{m_-}_-, \p^1_{\ve_{k+1}}],
$$
and the statement  follows from Lemma \ref{m+n-} for general $k$.
\end{proof}
Let $V_\bt\subset V^+_\bt\op V^-_{\bt'}\oplus\C u_\bt$ with $\bt\in   \bar \Pi_{ \l}$  be the $\l$-module generated by $x_\bt$
(the element $u_\bt\in \c$ can be non-zero only if $V^+_\bt$ and $V^-_{\bt'}$ are trivial $\l$-modules).
We define $\k^0=\mathfrak{t}+\l$ and
$$
\k^1=\sum_{\bt\in \Pi_{\g/\l}}V_{\bt}+\k^0,
\quad
\k^{m+1}=[\k^m,\k^1], \quad m>0.
$$
The $\l$-modules  $\k^m$ form an ascending filtration of $\k$.
Furthermore, condition (\ref{2nd-cond}) implies
$$
[y_\al,x_\bt]=\bigl((\al,\bt)-(\tilde \al,\bt)\bigr) e_\bt+c_\bt  \bigl(-(\al,\tilde \bt)+(\tilde \al,\tilde \bt)\bigr) f_{\tilde\bt}
\propto x_{\bt}
$$
for all $\al,\bt \in \bar \Pi_{ \l}$. The rightmost implication holds because the  right-hand side of the equality simply vanishes if $\bt$ is orthogonal to $\Pi_\l$
and $\bt=\tilde \bt$.  We conclude that  the filtration $\k^0\subset \k^1\subset \ldots \subset \k$ is $\mathfrak{t}$- and therefore $\k^0$-invariant.

\begin{propn}
\label{ps-sym=sup-sph}
  Pseudo-symmetric Lie superalgebras are spherical.
\end{propn}
\begin{proof}
Observe that $\k^0\subset \p^0$ and $\k^1+\p^1_-=\p_-^1+\p^1_+$ by construction. Suppose that we have proved an equality
$ \k^m+\p^m_-=\p_-^m+\p^m_+$ for $m\geqslant 1$.
For any sequence $(x_{\mu_i})_{i=1}^{m+1}$ we have
$$
[\ldots[x_{\mu_1},x_{\mu_2}],\ldots, x_{\mu_{m+1}}]=[\ldots[e_{\mu_1},e_{\mu_2}],\ldots, e_{\mu_{m+1}}]+c[\ldots[f_{\tau(\mu_1)},f_{\tau(\mu_2)}],\ldots, f_{\tau(\mu_{m+1})}]+\ldots,
$$
where $c\in \C$ and the omitted terms are in $\p^m_++\p^m_-$, by Corollary \ref{cor_+_-}.
By the induction assumption, the above expression simplifies  to
$$
=
[\ldots[e_{\mu_1},e_{\mu_2}],\ldots, e_{\mu_{m+1}}]\mod \k^m+\p^{m+1}_-.
$$
Since the filtration is $\l$-invariant, and all
$[\ldots[e_{\mu_1},e_{\mu_2}],\ldots, e_{\mu_{m+1}}]$ generate $\p^{m+1}_+\mod \p^m_+$ we conclude that
 $\k^{m+1}+\p^{m+1}_-\supset \p^{m+1}_+$ and $\k^{m+1}+\p^{m+1}_-\supset \p^{m+1}_+ +\p^{m+1}_-$.
 This implies
$$
\k+\b_-=(\k+\l)+\b_-=\k+(\l+\b_-)=\k+\p_-\supset \p_++\p_-=\g.
$$
The proof is complete.
\end{proof}
\noindent
In the next section we address  the question when the subalgebra $\k\subset \g$ is proper for a given pseudo-symmetric  triple,
in the special symmetric grading that we have fixed.
 \subsection{Decorated diagrams and selection rules}
\label{SecDDD-SR}
Like in the non-graded case the permutation $\tau$ entering a pseudo-symmetric triple $(\g,\l,\tau)$
is an involutive even automorphism of the Dynkin diagram coinciding with $-w_\l$ on $\Pi_\l$.
Such triples can be visualized via decorated Dynkin diagrams similar to non-graded pseudo-symmetric triples considered in  \cite{RV}.
In the graded setting, we use black colour for nodes in $\Pi_\l$ and white from $\bar \Pi_\l$ regardless of their parity.

In this section, we work out  criteria for a decorated diagram when the subalgebra $\k$ equals $\g$ for all values
of the mixture parameters.
Recall that in the non-graded case there is a  selection rule that is formulated as follows.
Let $C(\bt)$ denote  the union of connected components of $\Pi_\l$ that are  connected to  $\{\bt,\tau(\bt)\}$. Then
diagrams with
\be
\label{RVSR}
C(\bt)\cup \{\bt,\tau(\bt)\}\simeq
\begin{picture}(35,10)
\put(5,3){\circle*{3}}
\put(30,3){\color{black}\circle{3}}
 \put(7,3){\line(1,0){21}}
 \put(1,6){$\small \al$} \put(30,6){$\small \bt$}
 \end{picture}
\ee
should be ruled out since $\k=\g$ in this case  \cite{RV}.

Graded selection rules are more intricate. They  amount to a set of lemmas
which guarantee that   under certain conditions on $\bt\in \bar \Pi_\l$ the Lie superalgebra $\g^{(\bt)}=\Span\{e_\bt,f_\bt,h_\bt\}$
 along with $\l$-modules it generates  lies  in $\k$. That will be further used
to prove the equality $\k=\g$ in Proposition \ref{SR=>Triv}.
Now let us demonstrate that the only non-graded selection rule fails to be definite in the super case.
\begin{lemma}
\label{non-grad-sel-rule}
Suppose that a decorated Dynkin diagram is such that  (\ref{RVSR}) holds for some $\bt\in \bar \Pi_l$.
 Then  $\g^{(\bt)}\subset \k$ unless $\bt$ is odd and $\al$ is even.
 \end{lemma}
\begin{proof}
Suppose that  (\ref{RVSR}) is fulfilled with some $\bt\in\bar \Pi_\l$ and $\al\in \Pi_\l$
First of all observe that $\al$ cannot be odd, because otherwise $w_\l$ will not be even, by Proposition \ref{odd-Levi-iso}
(that will also violate the condition (\ref{1nd-cond})).

Furthermore, for even $\al$, we have $\tilde \bt=\bt+\al$.
 With $x_\bt=e_\bt+c_\bt [f_\al,f_\bt]\in \k$, we define
$$
x_{\tilde\bt}=[e_\al,x_\bt]=[e_\al,e_\bt]+c_\bt[h_\al,f_\bt]=[e_\al,e_\bt]+c_\bt f_\bt\in \k.
$$
Since $(\al,\bt)\not=0$, there is no term $\grave{c}_\bt u_\bt$ in $x_\bt$. Then
$$
[x_\bt,x_{\tilde\bt}]=[e_\bt+c_\bt [f_\al,f_\bt], [e_\al,e_\bt]+c_\bt f_\bt]=c_\bt\bigl( h_\bt+(-1)^{|\bt|} h_\bt+(-1)^{|\bt|} h_\al\bigr)\in \k.
$$
Since $h_\al\in \k$, we conclude that $h_\bt\in \k$ and therefore $e_\bt,f_\bt\in \k$ if and only if  $\bt$ is even.
\end{proof}

It follows from  the proof of this lemma  that the diagram
$\begin{picture}(35,10)
\put(5,3){\circle*{3}}
\put(30,3){\color{black}\circle{3}}
 \put(7,3){\line(1,0){21}}
\put(3,7){$\al$}\put(28,7){$\bt$}
 \end{picture}$
with even  $\al$ and odd $\bt$ gives rise to a 5-dimensional  spherical  subalgebra
$\k=\Span\{e_\al,h_\al,f_\al,x_\bt,x_{\tilde \bt}\}$ in $\g=\g\l(2|1)$.
This $\k$ is beyond the scope of the current study because this grading of $\g\l(2|1)$ is not symmetric.
As a subdiagram, it nevertheless appears, for instance, in
\begin{center}
 \begin{picture}(50,40)

\put(10,20){\circle*{3}}
\put(11.5,21.5){\line(3,2){25}}
\put(11.5,18.5){\line(3,-2){25}}
\put(37.5,38){\framebox(3,3)}
\put(37.5,0){\framebox(3,3)}


\put(40,3){\line(0,1){35}}
\put(38,3){\line(0,1){35}}
 \end{picture}
\quad
 \begin{picture}(50,40)
\put(10,20){\circle*{3}}
\put(11.5,21.5){\line(3,2){25}}
\put(11.5,18.5){\line(3,-2){25}}
\put(37.5,38){\framebox(3,3)}
\put(37.5,0){\framebox(3,3)}
\qbezier(44,6)(49,20)(44 ,34) \put(45,9){\vector(-1,-3){2}} \put(45,31){\vector(-1,3){2}}


\put(40,3){\line(0,1){35}}
\put(38,3){\line(0,1){35}}
 \end{picture}
 \end{center}
with two grey odd  tail nodes. These diagrams yield non-trivial pairs $(\g,\k)$ with  $\g=\o\s\p(2|4)$. They  can be extended further to the left with alternating white and black nodes, cf. (\ref{C-type-diag0}) and (\ref{C-type-diag}).

Similarly to shapes of graded Dynkin diagrams we consider shapes of decorated diagrams by discarding grading.
In other words, a graded decorated diagram is obtained from its shape by adding the grading datum.
\begin{propn}
 Let $\Dc$ be a decoration of a Dynkin diagram $D$.  Then the shape of $\Dc$ is a decorated shape of $D$.
\end{propn}
\begin{proof}
We can define the Weyl operator in the non-graded case as we did in in the graded, cf. Remark \ref{Rem-Weyl-op}.
Then for each $\Pi_\l$, the operator  $w_\l$ is the same as in the non-graded shape. Now notice that the  automorphism $\tau$ of a
graded diagram is an automorphism of its shape that coincides
with $-w_\l$ on $\Pi_\l$.
\end{proof}
Proposition \ref{odd-Levi-iso} states that a grey odd root cannot constitute a connected component of $\Pi_\l$.
It turns out that such a node cannot be isolated from $\Pi_\l$ in all cases.

\begin{lemma}
\label{isolated odd}
  Suppose that a decorated Dynkin diagram contains a subdiagram
\be
\label{ISO-ODD}
\begin{picture}(35,10)
\put(5,3){\circle{3}}
\put(30.5,1.5){\framebox(3,3)}
\multiput(7,3)(6,0){4}{\line(1,0){3}}
 \put(1,6){$\small \al$} \put(30,6){$\small \bt$}
 \end{picture}
\ee
with $\al,\bt \in \bar \Pi_\l$, where $\al$ is of any parity, a grey odd node $\bt=\tau(\bt)$ is isolated from $\Pi_\l$ and
  $(\bt,\al)\not =0$.  Then $\g^{(\al)}+\g^{(\bt)}\subset \k$.
\end{lemma}
\begin{proof}
Indeed, we conclude that $h_\bt\in \k$ because
$$
[x_\bt,x_\bt]=x_\bt x_\bt+x_\bt x_\bt=[e_\bt,c_\bt f_\bt]+[c_\bt f_\bt,e_\bt]=2c_\bt h_\bt.
$$
Observe that $(\bt,\tilde \al)=\bigl(\bt,\tau(\al)+\ldots \bigr)=\bigl(\tau(\bt),\tau(\al)\bigr)=(\bt,\al)$, where we have suppressed
 terms from $\Z\Pi_\l$. Commuting $h_\bt$ with $x_{\al}=e_{\al}+c_{\al}f_{\tilde\al}+\grave{c}_\al u_\al$ we find
$$
[h_\bt,x_{\al}]=(\bt,\al)(e_{\al}- c_{\al}f_{\tilde\al})\in \k.
$$
Thus we conclude that $e_{\al}\in \k$, $f_{\tilde\al}\in \k$, and therefore $f_\al,h_\al\in \k$, if $\grave{c}_\al u_\al=0$.
Finally, $h_\al$ splits $x_\bt$, so $e_\bt,f_\bt\in \k$ either.

Now suppose that $\grave{c}_\al u_\al\not =0$. Then $\tilde \al=\al$, and, since $h_\bt\in \k$ and $(\al,\bt)\neq 0$,  we have a system
$$
\left\{
\begin{array}{ccc}
e_\al+c_\al f_\al+\grave{c}_\al u_\al&=&0\mod  \k,
\nn\\
e_\al-c_\al f_\al &=&0\mod \k,\nn
\end{array}
\right.
\quad\Leftrightarrow\quad
\left\{
\begin{array}{cccl}
e_\al+c_+u_\al&=0& \mod\k, & c_+=\grave{c}_\al/2,\\
f_\al+c_-u_\al&=0& \mod\k, & c_-=\grave{c}_\al/2 c_\al.
\end{array}
\right.
$$
Commuting these elements with $h_\bt\in\k$, we obtain $e_\al,f_\al\in \k$ and hence $h_\al\in \k$.
\end{proof}
\begin{lemma} \label{le-b-d}
Suppose that decorated Dynkin diagram contains a subdiagram
\be
\label{4NODES}
\begin{picture}(350,15)
\put(160,3){\circle*{3}}
\put(161.5,3){\line(1,0){27}}

\put(188.5,1.5){\framebox(3,3)}

\put(191.5,3){\line(1,0){24}}
\put(217,3){\circle*{3}}
\put(242,3){\circle{3}}

\multiput(217,3)(6,0){4}{\line(1,0){3}}

\put(155,8){$\al$}
\put(185,8){$\bt$}

\put(212,8){$\gm$}
\put(242,8){$\sigma$}
 \end{picture}
\ee
where $\al$ and $\gm$ are even, $\bt$ is odd, and $\gm$ with $\si$ form an arbitrary  connected even diagram of rank 2. Then $\g^{(\bt)}+\g^{(\si)}\subset \k$.
\end{lemma}
\begin{proof}
By the assumption, $\tau(\bt)=\bt$ and $\tau(\si)=\si$.
 Choose the grading such that $(\al,\bt)=1=-(\bt,\gm)$.
We can set $\grave{c}_\bt=0$ and $\grave{c}_\si=0$ because $\bt$ and  $\si$ are not isolated from $\Pi_\l$.
Evaluating commutators of  $e_\al$ and  $e_\gm$ with $x_\bt=e_\bt+c_\bt [[f_\al,f_\bt],f_\gm]\in \k$, we obtain
$x_{\tilde \bt}=c_{\bt}f_\bt- [[e_\al,e_\bt],e_\gm]\in \k$. Then
$[x_\bt,x_{\tilde \bt}]=c_\bt(2 h_\bt+h_\al+h_\gm)=c_\bt h_\dt\in\k.$
Since
$$(\dt,\si)= (\gamma,\si)\not =-(\dt,\tilde \si)=-\Bigl(2\bt+\gm,\si-\frac{2(\gm,\si)}{(\gm,\gm)}\gm+\ldots\Bigl)=-4\frac{(\gm,\si)}{(\gm,\gm)}+(\gm,\si),$$
 the element $h_\dt$  splits $x_{\si}=e_\si+c_\si f_{\tilde\si}$ and therefore $x_{\tilde \si}\propto  f_\si+ c_{-\si} e_{\tilde \si}\in \k$ with some $c_{-\si }\in \C^\times$.
Hence $e_{\si},f_{\si}$, and $h_\si\in \k$. The element $h_\si$ splits $x_\bt$ and $x_{\tilde \bt}$ because
$(\si,\bt)=0$ and $-(\si,\tilde \bt)=-(\si,\gm)\neq 0$.
Therefore $e_{\bt},f_{\bt},h_\bt\in \k$ either.
\end{proof}
\noindent
The next lemma addresses diagrams with even orthogonal shape.
\begin{lemma}
\label{sel-rul-d}
 Suppose that a decorated Dynkin diagram of type $\g=\o\s\p(2\n|2\m)$ with rank $n$  contains one of the subdiagrams
\be
\label{D-TAIL}
\begin{picture}(250,50)

\put(30,20){\circle*{3}}
\put(31.5,20){\line(1,0){27}}

\put(58.5,18.5){\framebox(3,3)}
\put(61.5,20){\line(1,0){27}}

\put(90,20){\circle*{3}}
\put(91.5,21.5){\line(3,2){25}}
\put(91.5,18.5){\line(3,-2){25}}
\put(117.5,39){\circle{3}}
\put(117.5,1){\circle{3}}

\put(30,25){$\al$}
\put(60,25){$\bt$}
\put(87,25){$\gm$}
\put(122,38){$\si=\al_{n-1}=\al_m$}
\put(122,0){$\al_n=\si'$}

\qbezier(120,6)(125,20)(120 ,34) \put(121,9){\vector(-1,-3){2}} \put(121,31){\vector(-1,3){2}}
 \end{picture}
\begin{picture}(130,50)

\put(0,20){\circle*{3}}
\put(1.5,20){\line(1,0){27}}

\put(28.5,18.5){\framebox(3,3)}
\put(31.5,21.5){\line(3,2){25}}
\put(31.5,18.5){\line(3,-2){25}}
\put(57.5,39){\circle{3}}
\put(57.5,1){\circle*{3}}

\put(-5,25){$\al$}
\put(25,25){$\bt$}
\put(62,0){$\gm$}
\put(62,38){$\si=\al_{n}=\al_m$}
 \end{picture}
\ee
 where $\al$, $\gm$, and $\si$ are even and $\bt$ is odd.
 Then $\g^{(\bt)}+\g^{(\si)}+\g^{(\si')}\subset \k$.
\end{lemma}
\begin{proof}
The reasoning is similar to the proof of Lemma \ref{le-b-d}.
We will only indicate the differences.
First we assume that $\grave{c}_\si=0$ for the diagram on the right.
  By the assumption, $\tau(\bt)=\bt$ and $\tau(\si)=\si'$ for left diagram and $\tau(\si)=\si$ for the right one.
As before, we argue that
$h_\dt$ with $\dt =2 \bt+\al+\gm$ is in $\k$.
Further we find that
$$(\dt,\si)= (\gamma,\si)=-1\not =-(\dt,\tilde \si)=-(2\bt+\gm,\al_n+\gm)=1\ \mathrm{ (for\ left\ diagram)},$$
$$(\dt,\si)= (2\bt,\si)=-2\not =-(\dt,\tilde \si)=-(2\bt,\al_n)=2\ \mathrm{ (for\ right\ diagram)}.$$
As a consequence,  the element $h_\dt$  splits $x_{\si}$,  $x_{\tilde \si}$, and $x_{\bt}$.
For the left diagram, the rest of the proof is essentially the same as in Lemma \ref{le-b-d}.
For the right diagram, it can be proved that $x_\bt$ and $x_{\tilde \bt}$ are separated by the element
$h_\si\in \k$ because
$$
 (\si,\bt) = 1\neq -(\si,\tilde \bt) = -(\si,\al + \bt + \gamma) =- (\si,\bt) = -1.
$$
 In both cases, $e_{\bt}, f_{\bt}, h_{\bt} \in \k$.

Now suppose that the term $\grave{c}_\si u_\si$ is present in $x_\si$ for the diagram on the right. We proceed
as in the proof of Lemma \ref{isolated odd} and demonstrate that $\g^{(\si)}\in \k$. Then we complete the proof as before.
\end{proof}

  \begin{definition}
We call a triple  $(\g,\l,\tau_\l)$ and the corresponding decorated Dynkin diagram trivial if the subalgebra $\k$
they generate coincides with $\g$ for all values of mixture parameters $c_\al\in \C^\times$, $\grave c_\al \in \C$,  $\al\in \bar \Pi_\l$.
\end{definition}

We will say that a decorated diagram $D$ violates selection rules if either
 $D$ contains an even  subdiagram (\ref{RVSR}) or one of the subdiagrams (\ref{ISO-ODD}), (\ref{4NODES}), (\ref{D-TAIL}).
It is important to note  that we allow for arbitrary orientation of these subdiagrams while the orientation of the total diagram is fixed as in Section \ref{SecBasQSG}.

For a diagram $D$, we define a subdiagram $D_0\subset D$ of $\g\l$-type by throwing away the tail. Specifically, we remove from  $D$
the two last (rightmost) nodes if the shape $\tilde D$ of $D$ is of  even orthogonal type, and the last node if $\tilde D$
 is either odd orthogonal or symplectic. For general linear $D$, we set $D_0=D$.
\begin{lemma}
  If $D$ violates the selection rule, then $\tau|_{D_0}=\id$ in all cases.
\label{tau=id}
\end{lemma}
\begin{proof}
  We need to consider only the case of general linear $\g$. Then only subdiagrams (\ref{RVSR}), (\ref{ISO-ODD}), or (\ref{4NODES})
may occur in $D$. For the subdiagrams  (\ref{RVSR}) and  (\ref{4NODES}) the statement is obvious.
In the case of   (\ref{ISO-ODD}), the symmetric grading under consideration requires the presence of two odd nodes. Therefore $\bt$ cannot be $\tau$-fixed if $\tau$ reverts  $D$. Thus $\tau$ is identical in all cases.
\end{proof}

\begin{propn}
\label{SR=>Triv}
If a decorated Dynkin diagram violates selection rules, then it is trivial.
\end{propn}
\begin{proof}
Denote by  $L\subset D$ a  subdiagram that violates the selection rules. Split the set of nodes of total diagram
to a disjoint union  $D=D^l\cup D^r$, where $D^l$ comprises the nodes on the left of $L$.
In the case when the shape $\tilde D$ of $D$ is even orthogonal, and $L=\{\al,\bt\}$ is in the tail of $D$, we include all  the three tail nodes  in  $D^r$.
The further proof will be done in two steps. First, moving from $L$ rightward, we demonstrate that  $\g^{(\mu)}\subset \k$ for all $\mu\in D^r$.
Then we proceed leftward and prove that for the nodes from $D^l$.

It follows that $\tau$ is identical on $\{\al,\bt\}\subset L$.
If $\tilde D$ is even orthogonal and  $L=\{\al,\bt\}$ is in the tail, then
$g^{(\mu)}\subset \k$ for all $\mu\in D^r$ by Lemmas \ref{non-grad-sel-rule} and \ref{isolated odd}.
We arrive to the  same conclusion if $L$ is (\ref{D-TAIL}), by Lemma \ref{sel-rul-d}: then
  $L=D^r$.

Consider the case when  $L\subset D_0$ assuming that $L$ is either (\ref{RVSR}) or (\ref{ISO-ODD}).
Put $L_1=L$ and denote by  $\l_1$ the subalgebra with roots in $L_1$. We have an inclusion $\l_1\subset \k$.
Suppose that we have constructed $L_k\subset D_0^r$ and $\l_k\subset \k$ for $k\geqslant 1$.
If there is a node $\mu\in D_0^r\cap \bar \Pi_\l$ on the right of $L_k$ that is
connected to $L_k$, then $\g^{(\mu)}\subset \k$, because the $\l_k$-module generated by the root vector $e_\mu$ is not self-dual (mind that
$\tau$ is identical on $D_0$ by Lemma \ref{tau=id}).
Let $C_\mu\subset D_0^r\cap \Pi_\l$ be the connected component that is connected to $\mu$ from the right, that is, $C_\mu\cap L_k=\varnothing$.
Then we set $L_{k+1}=L_k\cup \{\mu\}\cup C_\mu$ and define $\l_{k+1}=\l_k+\g^{(\mu)}+\sum_{\nu\in C_\mu}\g^{(\nu)}$.

It is clear that $L_k=D_0^r$ and  for sufficiently large $k$. Finally, by considering a tail node $\mu$ of the initial diagram $D^r$
we conclude that $\g^{(\mu)}\subset \k$ for all $\mu\in D^r$.

Now we start moving leftward to process nodes from $D^l$. This time we include in consideration the diagram $L$ as in (\ref{D-TAIL}).
Suppose there is a node $\mu\in D^l\cap \bar \Pi_\l$. We can assume that it is the rightmost among them.
Then the mixed generator $x_\mu=f_\mu+c_\mu e_{\tilde \mu}+\grave c_\mu u_\al$ is split by the subalgebra $\mathfrak{m}$ comprising the roots on the right of $\mu$,
because $f_\mu$, $e_{\tilde \mu}$, and $u_\mu$ transform differently under $\mathfrak{m}$.
\end{proof}
\noindent
As a consequence, we derive the following  restriction on the position of the odd root in some ortho-symplectic diagrams.
\begin{corollary}
\label{Levi-Gr-Boxes}
Suppose that $\m<m $. Then the subalgebra $\k$ corresponding to a diagram
(\ref{OSP-I-odd})-(\ref{OSP-I-even-2})  exhausts all of ortho-symplectic Lie superalgebra $\g$.
\end{corollary}
\begin{proof}
These diagrams correspond to K-matrices of the form $A$ and $B$ from Theorem \ref{s-graded}.
  For a K-matrix of type $A$ such a diagram is ruled out by Lemma \ref{isolated odd}, while of type $B$ by Lemma \ref{le-b-d}.
\end{proof}
\subsection{Graded Satake diagrams}
\label{Sec-Gr-Satake}
To avoid conflict with the conventional notation of graded Dynkin diagrams, we do not use coloured nodes to indicate their parity.
Black nodes will designate roots from $\Pi_\l$, white the nodes from $\bar \Pi_\l$,  and the odd nodes will be  denoted with square.

Our selection rules allow for the following decorated graded Dynkin diagrams with odd nodes in $\bar \Pi_\l$:
\be
  \begin{picture}(150,10)
\put(-4,1){$\scriptstyle($}\put(2,1){$\scriptstyle)$}
\put(0,3){\circle{3}}
 \put(2,3){\line(1,0){5}}
 \put(9,0){$\cdots$}
  \put(23,3){\line(1,0){5}}
\put(32,3){\line(1,0){8}}

\put(28.5,1.5){\framebox(3,3)}
\put(33,3){\line(1,0){7}}
\put(42,3){\circle*{3}}
\put(43,3){\line(1,0){5}}
\put(51,0){$\cdots$}

\put(68,1){$\scriptstyle($}\put(83,1){$\scriptstyle)$}
\put(66,3){\line(1,0){5}}
\put(72,3){\circle{3}}
\put(74,3){\line(1,0){7}}
\put(82,3){\circle*{3}}

\put(83,3){\line(1,0){8}}
\put(91,1.5){\framebox(3,3)}

\put(94,3){\line(1,0){5}}
\put(101,0){$\cdots$}
\put(116,3){\line(1,0){5}}
\put(122,3){\circle{3}}
\put(118,1){$\scriptstyle($}\put(124,1){$\scriptstyle)$}

 \end{picture}
\quad
\begin{picture}(160,10)
\put(-2,1){$\scriptstyle($}\put(14,1){$\scriptstyle)$}

\put(2,3){\circle*{3}}
\put(3,3){\line(1,0){8}}
\put(12,3){\circle{3}}
\put(13,3){\line(1,0){5}}

\put(21,0){$\cdots$}

\put(36,3){\line(1,0){5}}

\put(42,3){\circle*{3}}

\put(43,3){\line(1,0){8}}
\put(51,1.5){\framebox(3,3)}
\put(54,3){\line(1,0){8}}
\put(60,1){$\scriptstyle($}\put(65,1){$\scriptstyle)$}
\put(64,3){\circle{3}}
\put(66,3){\line(1,0){5}}
\put(73,0){$\cdots$}
\put(88,3){\line(1,0){5}}
\put(94,1.5){\framebox(3,3)}
 \put(97,3){\line(1,0){8}}

\put(106,3){\circle*{3}}
\put(101,3){\line(1,0){5}}
\put(115,0){$\cdots$}

\put(130,3){\line(1,0){5}}
\put(132,1){$\scriptstyle($}\put(147,1){$\scriptstyle)$}
\put(136,3){\circle{3}}
\put(138,3){\line(1,0){7}}
\put(146,3){\circle*{3}}

\end{picture}
\label{ANOM-GL}
\ee
of $\g\l$-type and
\be
\label{ANOM-OSP}
\begin{picture}(130,10)
\put(-4,1){$\scriptstyle($}\put(2,1){$\scriptstyle)$}
\put(0,3){\circle{3}}
\put(2,3){\line(1,0){5}}
\put(9,0){$\cdots$}
\put(24,3){\line(1,0){5}}
\put(30,1.5){\framebox(3,3)}
\put(33,3){\line(1,0){7}}
\put(42,3){\circle*{3}}
\put(43,3){\line(1,0){7}}
\put(51,3){\circle{3}}
\put(52,3){\line(1,0){5}}
\put(60,0){$\cdots$}
\put(82,3){\circle*{3}}
\put(38,1){$\scriptstyle[$}\put(54,1){$\scriptstyle]$}
\put(83,3){\line(1,0){8}}
\put(92,0){$\cdots$}
\put(75,3){\line(1,0){5}}
\put(112,3){\circle*{3}}
\put(106,3){\line(1,0){5}}

\put(112,3){\circle*{3}}
\put(112,3){\line(1,0){12}}
\put(125,3){\circle*{6}}
\end{picture}
\quad\quad
\begin{picture}(50,10)
\put(0,3){\circle{3}}
\put(-4,1){$\scriptstyle($}\put(1,1){$\scriptstyle)$}
\put(1.5,3){\line(1,0){9}}\put(13,0){$\cdots$} \put(29.5,3){\line(1,0){6}}

 \put(36,1.5){\framebox(3,3)}

\put(39.5,3.5){\line(1,1){10}}\put(39.5,3.2){\line(1,-1){10}}
\put(50.5,14){\circle{3}}\put(50.5,-8){\circle*{3}}
\end{picture}
\quad\quad
  \begin{picture}(100,10)

\put(2,3){\circle*{3}}
\put(3,3){\line(1,0){5}}
\put(11,0){$\cdots$}

\put(28,1){$\scriptstyle($}\put(43,1){$\scriptstyle)$}
\put(26,3){\line(1,0){5}}
\put(32,3){\circle{3}}
\put(34,3){\line(1,0){7}}
\put(42,3){\circle*{3}}

\put(43,3){\line(1,0){8}}
\put(51,1.5){\framebox(3,3)}
\put(54,3){\line(1,0){8}}
\put(64,3){\circle{3}}
\put(60,1){$\scriptstyle($}\put(65,1){$\scriptstyle)$}
\put(67,3){\line(1,0){5}}
\put(73,0){$\cdots$}

\put(88,3){\line(1,0){5}}

\put(96,3){\circle{6}}

 \end{picture}
\ee
 of $\o\s\p$- and $\s\p\o$-types, where the big circles on the right stand for even tail.
The subdiagrams enclosed in the brackets mean the period including zero occurrence for $(\>\>)$ and at least once  for $[\>\>]$.

Observe that shapes of (\ref{ANOM-GL}) and (\ref{ANOM-OSP}) are not admissible generalized Satake diagrams from \cite{RV}.
Here is yet another diagram that shares this property:
\be
\begin{picture}(100,10)
\put(0,3){\circle*{3}}\put(1.5,3){\line(1,0){12}}\put(15,3){\circle{3}}\put(16.5,3){\line(1,0){12}}\put(30,3){\circle*{3}}

\put(31.5,3){\line(1,0){9}}\put(43,0){$\cdots$} \put(59.5,3){\line(1,0){9}}

\put(70,3){\circle*{3}}\put(71.5,3){\line(1,0){12}}\put(85,3){\circle{3}}\put(86.5,3){\line(1,0){12}}\put(100,3){\circle*{3}}

\put(100,3.5){\line(1,1){11}}\put(100,3.5){\line(1,-1){11}}
\put(111,13){\framebox(3,3)}\put(111,-8){\framebox(3,3)}
\put(111,13){\line(0,-1){18}}\put(114,13){\line(0,-1){18}}
\label{C-type-diag0}
\end{picture}
\ee
It may be therefore attributed to the family with white tail depicted on the right in (\ref{ANOM-OSP}).

All other admissible diagrams have non-graded generalized  Satake shape; they are listed below.
We arrange them in   families of diagrams of  the same shape,
indicating  possible position of odd roots  with square nodes.
Recall that the number of such nodes and their type have been fixed for each  $\g$ by the minimal symmetric grading of the underlying vector space.
In particular, there are two symmetrically allocated odd
simple roots for  $\g\l(N|2\m)$ and $\o\s\p(2|2\m)$, and exactly one such root otherwise.
 An arc connects the root $\tau(\al)$ with $\al\in \bar \Pi_\l$ if they are distinct, in the usual way.
Black nodes depict simple roots from $\Pi_\l$, whereas white nodes label elements of $\bar \Pi_\l$.
\be
\begin{picture}(120,30)
\put(-1.5,1.5){\framebox(3,3)}\put(1.5,3){\line(1,0){12}}\put(13.5,1.5){\framebox(3,3)}
\put(16.5,3){\line(1,0){9}}\put(28,0){$\cdots$} \put(44.5,3){\line(1,0){9}}
\put(53.5,1.5){\framebox(3,3)}\put(56.5,3){\line(1,0){12}}\put(68.5,1.5){\framebox(3,3)}

\put(-1.5,27.5){\framebox(3,3)}\put(1.5,29){\line(1,0){12}}\put(13.5,27.5){\framebox(3,3)}
\put(16.5,29){\line(1,0){9}}\put(28,26){$\cdots$} \put(44.5,29){\line(1,0){9}}
\put(53.5,27.5){\framebox(3,3)}\put(56.5,29){\line(1,0){12}}\put(68.5,27.5){\framebox(3,3)}

\put(71.5,29){\line(1,-1){12}}
\put(71.5,3){\line(1,1){12}}\put(84,16){\circle{3}}

\qbezier(0,6)(-5,16)(0,26)\qbezier(15,6)(10,16)(15,26)\qbezier(55,6)(50,16)(55,26)\qbezier(70,6)(65,16)(70,26)
\end{picture}
\quad
\begin{picture}(80,30)
\put(-1.5,1.5){\framebox(3,3)}\put(1.5,3){\line(1,0){12}}\put(13.5,1.5){\framebox(3,3)}
\put(16.5,3){\line(1,0){9}}\put(28,0){$\cdots$} \put(44.5,3){\line(1,0){9}}
\put(53.5,1.5){\framebox(3,3)}\put(56.5,3){\line(1,0){12}}\put(68.5,1.5){\textcolor{black}{\rule{3pt}{3pt}}}
\put(70,28){\line(0,-1){5}}\put(70,4){\line(0,1){5}}
\put(68.5,11){\vdots}
\put(-1.5,27.5){\framebox(3,3)}\put(1.5,29){\line(1,0){12}}\put(13.5,27.5){\framebox(3,3)}
\put(16.5,29){\line(1,0){9}}\put(28,26){$\cdots$} \put(44.5,29){\line(1,0){9}}
\put(53.5,27.5){\framebox(3,3)}\put(56.5,29){\line(1,0){12}}\put(68.5,27.5){\textcolor{black}{\rule{3pt}{3pt}}}

\qbezier(0,6)(-5,16)(0,26)\qbezier(15,6)(10,16)(15,26)\qbezier(55,6)(50,16)(55,26)
\end{picture}
\label{GL-I}
\ee
\be
\label{SPO-I}
\begin{picture}(130,10)
\put(0,3){\circle{3}}
\put(1.5,3){\line(1,0){9}}\put(13,0){$\cdots$} \put(29.5,3){\line(1,0){9}}
\put(40,3){\circle{3}}\put(41.5,3){\line(1,0){12}}\put(53.5,1.5){\framebox(3,3)}\put(56.5,3){\line(1,0){12}}\put(68.5,1.5){\textcolor{black}{\rule{3pt}{3pt}}}
\put(71.5,3){\line(1,0){9}}\put(83,0){$\cdots$} \put(99.5,3){\line(1,0){9}}\put(108.5,1.5){\textcolor{black}{\rule{3pt}{3pt}}}
\put(113,4){\line(1,0){11}}
\put(113,2){\line(1,0){11}}
\put(110.5,1){$\scriptstyle <$}
\put(125,3){\circle*{3}}
\end{picture}
\quad\quad
\begin{picture}(160,10)
\put(15,3){\circle*{3}}\put(16.5,3){\line(1,0){12}}\put(30,3){\circle{3}}
\put(31.5,3){\line(1,0){9}}\put(43,0){$\cdots$} \put(59.5,3){\line(1,0){9}}
\put(70,3){\circle*{3}}\put(71.5,3){\line(1,0){12}}\put(83.5,1.5){\framebox(3,3)}\put(86.5,3){\line(1,0){12}}\put(98.5,1.5){\textcolor{black}{\rule{3pt}{3pt}}}
\put(101.5,3){\line(1,0){9}}\put(113,0){$\cdots$} \put(129.5,3){\line(1,0){9}}\put(138.5,1.5){\textcolor{black}{\rule{3pt}{3pt}}}
\put(143,4){\line(1,0){11}}
\put(143,2){\line(1,0){11}}
\put(140.5,1){$\scriptstyle <$}
\put(155,3){\circle*{3}}
\end{picture}
\ee
\be
\label{OSP-I-odd}
\begin{picture}(125,10)
\put(0,3){\circle{3}}
\put(1.5,3){\line(1,0){9}}\put(13,0){$\cdots$} \put(29.5,3){\line(1,0){9}}
\put(40,3){\circle{3}}\put(41.5,3){\line(1,0){12}}\put(53.5,1.5){\framebox(3,3)}\put(56.5,3){\line(1,0){12}}\put(68.5,1.5){\textcolor{black}{\rule{3pt}{3pt}}}
\put(71.5,3){\line(1,0){9}}\put(83,0){$\cdots$} \put(99.5,3){\line(1,0){9}}\put(108.5,1.5){\textcolor{black}{\rule{3pt}{3pt}}}
\put(111,4){\line(1,0){11}}
\put(111,2){\line(1,0){11}}
\put(118.5,1){$\scriptstyle >$}
\put(123.5,1.5){\textcolor{black}{\rule{3pt}{3pt}}}
\end{picture}
\quad
\begin{picture}(155,10)
\put(15,3){\circle*{3}}\put(16.5,3){\line(1,0){12}}\put(30,3){\circle{3}}
\put(31.5,3){\line(1,0){9}}\put(43,0){$\cdots$} \put(59.5,3){\line(1,0){9}}
\put(70,3){\circle*{3}}\put(71.5,3){\line(1,0){12}}\put(83.5,1.5){\framebox(3,3)}\put(86.5,3){\line(1,0){12}}\put(98.5,1.5){\textcolor{black}{\rule{3pt}{3pt}}}
\put(101.5,3){\line(1,0){9}}\put(113,0){$\cdots$} \put(129.5,3){\line(1,0){9}}\put(138.5,1.5){\textcolor{black}{\rule{3pt}{3pt}}}
\put(141,4){\line(1,0){11}}
\put(141,2){\line(1,0){11}}
\put(148.5,1){$\scriptstyle >$}
\put(153.5,1.5){\textcolor{black}{\rule{3pt}{3pt}}}
\end{picture}
\quad
\quad
\begin{picture}(55,10)
\put(0,3){\circle{3}}
\put(1.5,3){\line(1,0){9}}\put(13,0){$\cdots$} \put(29.5,3){\line(1,0){9}}
\put(40,3){\circle{3}}
\put(41,4){\line(1,0){11}}
\put(41,2){\line(1,0){11}}
\put(48.5,1){$\scriptstyle >$}
\put(53.5,1.5){\framebox(3,3)}
\end{picture}
\ee
\be
\label{OSP-I-even}
\begin{picture}(140,10)
\put(15,3){\circle{3}}
\put(16.5,3){\line(1,0){9}}\put(28,0){$\cdots$} \put(44.5,3){\line(1,0){9}}
\put(55,3){\circle{3}}
\put(56.5,3){\line(1,0){12}}
\put(68.5,1.5){\framebox(3,3)} \put(71.5,3){\line(1,0){12}}
\put(83.5,1.5){\textcolor{black}{\rule{3pt}{3pt}}}
 \put(86.5,3){\line(1,0){9}}\put(98,0){$\cdots$} \put(114.5,3){\line(1,0){9}}
\put(123.5,1.5){\textcolor{black}{\rule{3pt}{3pt}}}

\put(125,3.5){\line(1,1){11}}\put(125,3.5){\line(1,-1){11}}
\put(137.5,15){\circle*{3}}
\put(137.5,-9){\circle*{3}}
\end{picture}
\quad\quad\quad\quad
\begin{picture}(140,10)
\put(0,3){\circle*{3}}
\put(1.5,3){\line(1,0){12}}
\put(15,3){\circle{3}}
\put(16.5,3){\line(1,0){12}}
\put(30,3){\circle*{3}}
\put(31.5,3){\line(1,0){9}}\put(43,0){$\cdots$} \put(59.5,3){\line(1,0){9}}

\put(68.5,1.5){\framebox(3,3)} \put(71.5,3){\line(1,0){12}}
\put(83.5,1.5){\textcolor{black}{\rule{3pt}{3pt}}}
 \put(86.5,3){\line(1,0){9}}\put(98,0){$\cdots$} \put(114.5,3){\line(1,0){9}}
\put(123.5,1.5){\textcolor{black}{\rule{3pt}{3pt}}}

\put(125,3.5){\line(1,1){11}}\put(125,3.5){\line(1,-1){11}}
\put(137.5,15){\circle*{3}}
\put(137.5,-9){\circle*{3}}
\end{picture}
\ee
\be
\label{OSP-I-even-2}
\begin{picture}(140,10)
\put(15,3){\circle{3}}
\put(16.5,3){\line(1,0){9}}\put(28,0){$\cdots$} \put(44.5,3){\line(1,0){9}}
\put(55,3){\circle{3}}
\put(56.5,3){\line(1,0){12}}
\put(68.5,1.5){\framebox(3,3)} \put(71.5,3){\line(1,0){12}}
\put(83.5,1.5){\textcolor{black}{\rule{3pt}{3pt}}}
 \put(86.5,3){\line(1,0){9}}\put(98,0){$\cdots$} \put(114.5,3){\line(1,0){9}}
\put(123.5,1.5){\textcolor{black}{\rule{3pt}{3pt}}}

\put(125,3.5){\line(1,1){11}}\put(125,3.5){\line(1,-1){11}}
\put(137.5,15){\circle*{3}}
\put(137.5,-9){\circle*{3}}
\put(136.5,13.5){\line(0,-1){21}}\put(138.5,13.5){\line(0,-1){21}}
\end{picture}
\quad\quad\quad\quad
\begin{picture}(140,10)
\put(0,3){\circle*{3}}
\put(1.5,3){\line(1,0){12}}
\put(15,3){\circle{3}}
\put(16.5,3){\line(1,0){12}}
\put(30,3){\circle*{3}}
\put(31.5,3){\line(1,0){9}}\put(43,0){$\cdots$} \put(59.5,3){\line(1,0){9}}

\put(68.5,1.5){\framebox(3,3)} \put(71.5,3){\line(1,0){12}}
\put(83.5,1.5){\textcolor{black}{\rule{3pt}{3pt}}}
 \put(86.5,3){\line(1,0){9}}\put(98,0){$\cdots$} \put(114.5,3){\line(1,0){9}}
\put(123.5,1.5){\textcolor{black}{\rule{3pt}{3pt}}}

\put(125,3.5){\line(1,1){11}}\put(125,3.5){\line(1,-1){11}}
\put(137.5,15){\circle*{3}}
\put(137.5,-9){\circle*{3}}
\put(136.5,13.5){\line(0,-1){21}}\put(138.5,13.5){\line(0,-1){21}}
\end{picture}
\ee
 \be
\label{OSP-I-even-2-flip}
\hspace{5pt}
 \begin{picture}(145,10)
\put(55,3){\circle{3}}
 \put(56.5,3){\line(1,0){9}}\put(68,0){$\cdots$} \put(84.5,3){\line(1,0){9}}
\put(95,3){\circle{3}}

\put(96,3.5){\line(1,1){11}}\put(96,3.5){\line(1,-1){11}}
\put(107,13){\framebox(3,3)}\put(107,-8){\framebox(3,3)}
\put(107,13){\line(0,-1){18}}\put(110,13){\line(0,-1){18}}
\qbezier(111,-7)(118,4)(111,14)
\end{picture}
 \quad
\begin{picture}(115,10)
\put(0,3){\circle*{3}}\put(1.5,3){\line(1,0){12}}\put(15,3){\circle{3}}\put(16.5,3){\line(1,0){12}}\put(30,3){\circle*{3}}

\put(31.5,3){\line(1,0){9}}\put(43,0){$\cdots$} \put(59.5,3){\line(1,0){9}}

\put(70,3){\circle*{3}}\put(71.5,3){\line(1,0){12}}\put(85,3){\circle{3}}\put(86.5,3){\line(1,0){12}}\put(100,3){\circle*{3}}

\put(100,3.5){\line(1,1){11}}\put(100,3.5){\line(1,-1){11}}
\put(111,13){\framebox(3,3)}\put(111,-8){\framebox(3,3)}
\put(111,13){\line(0,-1){18}}\put(114,13){\line(0,-1){18}}
\qbezier(115,-7)(122,4)(115,14)
\label{C-type-diag}
\end{picture}
\ee

\vspace{5pt}
\begin{definition}
Diagrams (\ref{GL-I})-(\ref{C-type-diag}) are said to be of type I.
  Diagrams (\ref{ANOM-GL}) and (\ref{ANOM-OSP}) along with (\ref{C-type-diag0}) will be referred to as of type II.
They are all called $\Z_2$-graded Satake diagrams.
\end{definition}
\noindent
We also extend the above  classification to the corresponding  spherical pairs.
Graded Satake diagrams are the only decorated Dynkin diagrams
that comply with the selection rules of the previous section.

\begin{conjecture}
  Graded Satake diagrams  are non-trivial.
\end{conjecture}

Diagrams of type I yield non-trivial $\k$ for certain values of mixture parameters because they correspond to K-matrices of type $A$ and $B$ from Theorem \ref{s-graded}.
That implies inequality $\dim \End_\k(V)>\dim \End_\g(V)$ for the basic module $V$, and therefore such $\k$ are proper in $\g$.

Further we present arguments in support that diagrams of type II are non-trivial either.
In the special case of (\ref{C-type-diag0}) the  coideal subalgebra centralizes a K-matrix of type $C$, cf. Section \ref{SecKmatC}.
We guess the following K-matrices for black tailed diagrams  in (\ref{ANOM-OSP}):
\be
K&=&\sum_{i=1}^{m}(\la+\mu)e_{ii}+\sum_{m<i<m'}\la e_{ii}+\sum_{i\leqslant \m }(y_{i}e_{i,i'}+y_{i'}e_{i',i})
\nn\\
&+&\sum_{\m<i <m\atop i-\m = 1\!\! \!\!\mod 2}z_i(e_{i,i'-1}-e_{i+1,i'})+z_{i'-1}(e_{i'-1,i}-e_{i',i+1}),\quad \mbox{where}
\label{K-black-tailed}
\ee
$$
y_i y_{i'}=-\la \mu ,\quad i\leqslant  \m, \quad z_i z_{i'-1}=-\la\mu,\quad \m<i< m, \quad \mu=\kappa_m\kappa_{m'}q^{-2\rho_{m+1}}\la.
$$
The parameter $m$ is the position of the rightmost white node in the Satake diagram while $\m$ is that of the odd root.
One can see that the shape of matrix $B$ from Theorem \ref{s-graded} may be viewed as  a degeneration of (\ref{K-black-tailed}).

The K-matrix for the white tailed diagrams in (\ref{ANOM-OSP}) is conjectured to be
\be
K&=&
\sum_{i=1}^n(\mu+\la)e_{ii}+\la e_{n+1,n+1}+\sum_{\m< i\leqslant n}(y_{i}e_{i,i'}+y_{i'}e_{i',i})
\nn\\
&+&
\sum_{i<\m \atop i= 1\!\!\!\!\mod 2}z_i(e_{i,i'-1}-e_{i+1,i'})+z_{i'-1}(e_{i'-1,i}-e_{i',i+1})
.
\label{K-white-tailed}
\ee
The term $\la e_{n+1,n+1}$ is present only for $\g=\o\s\p(2\n+1|2\m)$ (recall that $n$ stands for the rank $\m+\n$ of $\g$). The other parameters
are subject to the conditions
$$
\left\{
\begin{array}{ccl}
 y_i y_{i'}&=&-\la\mu,\quad \mbox{for}\quad \m< i\leqslant n , \\
z_i z_{i'-1}&=&-\la\mu,\quad \mbox{for}\quad i < \m,
\end{array}
\right.
\quad
\mu=
\left\{
\begin{array}{ccl}
  -\la, & \g= & \o\s\p(2\n|2\m), \\
  -q^{-1}\la, & \g =&\o\s\p(2\n+1|2\m),
\end{array}
\right.
$$
The parameters $\la$ and $\mu$ are eigenvalues of $K$. The expression (\ref{K-white-tailed}) is a generalization of the matrix $C$, which
is yet another justification for relating the diagram (\ref{C-type-diag0}) to  type II.

The K-matrix for the middle diagram in (\ref{ANOM-OSP}) is believed to be
\be
K&=&
\sum_{i=1}^n(\mu+\la)e_{ii}+\sum_{i=1}^{ n-2}(y_{i}e_{i,i'}+y_{i'}e_{i',i})
\nn\\
&+&
z_n(e_{n,n'-1}-e_{n+1,n'})+z_{n'-1}(e_{n'-1,n}-e_{n',n+1})
.
\label{K-half-tailed}
\ee
with arbitrary eigenvalues $\la$, $\mu$, and the relations $y_i y_{i'}=-\la \mu$, $z_n z_{n'-1}=-\la \mu$.
Conjugation with $\sum_{i\not =n,n'}e_{ii} + e_{n,n'}+e_{n',n}$ produces another K-matrix that corresponds to the tail flipped Satake diagram.

The above matrices cover all diagrams in  (\ref{ANOM-OSP}) apart from white tailed  with $\tau\not =\id$ for $\g=\o\s\p(2\m|2\n)$.
 The representation $\pi$ intertwines the automorphism of $U_q(\g)$  that  flips the tail root vectors with conjugation by the matrix
$\sum_{i\not =n,n'}e_{ii} + e_{n,n'}+e_{n',n}$.
The K-matrix for such diagrams  is conjectured to satisfy (\ref{REtw1}) and equal
\be
K&=&
\sum_{i=1}^{n-1}(\mu+\la)e_{ii}+\sum_{\m< i\leqslant n}(y_{i}e_{i,i'}+y_{i'}e_{i',i})
\nn\\
&+&
\sum_{i<\m \atop i= 1\!\!\!\!\mod 2}z_i(e_{i,i'-1}-e_{i+1,i'})+z_{i'-1}(e_{i'-1,i}-e_{i',i+1}),
\label{K-white-tailed-twisted}
\ee
where $\mu=-q^{-2}\la$, $y_n=y_{n'}=\la$,
$y_{i'}y_i=-\la\mu$, for $i=\m+1,\ldots n-1$, and  $z_i z_{i'-1}=-\la\mu,\quad \mbox{for}\quad i < \m$.

It is interesting to note that eigenvalues of the  K-matrix (\ref{K-white-tailed}) depend on two independent parameters for $\g=\s\p\o(2\n|2\m)$ and only  one parameter for $\g=\o\s\p(2\n|2\m)$ although $\o\s\p(2\n|2\m)\simeq \s\p\o(2\m|2\n)$. This may manifest that their quantum supergroups, which differ by the
choice of Borel subalgebra in $\g$, are not isomorphic as Hopf superalgebras.

The diagrams in  (\ref{ANOM-GL}) are related to the twisted RE in the form of (\ref{REtw}).
The basic representation intertwines the matrix supertransposition $A\mapsto A^t$, $A^t_{ij}=A_{ji}(-1)^{(|i|+|j|)|i|}$ with
a superalgebra  anti-automorphphism $\si\colon U_q(\g)\to U_q(\g)$ acting by
$$
e_\al\mapsto a_\al f_\al q^{h_\al}, \quad f_\al\mapsto  (-1)^{|\al|} a_\al^{-1} q^{-h_\al} e_\al, \quad q^{\pm h_\al}\mapsto q^{\pm h_\al},
$$
with an appropriate choice of $a_\al\in \C^\times$.
The K-matrix for the right diagrams  in (\ref{ANOM-GL}) corresponds to $\g=\g\l(N|2\m)$  and is believed to be
\be
K&=&
\sum_{i\leqslant \m \atop i= 1\!\!\!\!\mod 2}z_i(e_{i,i+1}-q e_{i+1,i})+z_{i'-1}(e_{i'-1,i'}-q e_{i',i'-1})
+\sum_{\m < i <\m'}y_ie_{ii},
\label{GL-right}
\ee
assuming $\m\in 2\N$.

The left   diagrams   in (\ref{ANOM-GL}) for $\g=\g\l(N|2\m)$ imply that  $N\in 2\N$. The K-matrix for it is conjectured to be
\be
K&=&
\sum_{i\leqslant \m}y_ie_{ii}+y_{i'}e_{i',i'}
+
\sum_{\m < i <\m'\atop i-\m= 1\!\!\!\!\mod 2}z_i(e_{i,i+1}-q^{-1} e_{i+1,i}).
\label{GL-left}
\ee
The parameters $y_i$ and $z_i$ in (\ref{GL-right}) and (\ref{GL-left}) are arbitrary.

The hypothetical K-matrices presented in this section cover all graded Satake diagrams of type II. They have been checked in low dimensions, and we expect that
they satisfy the appropriate RE in general.

  \subsection{Spherical pairs and   Reflection Equation}
Consider matrices $A, B, C$ from Theorem \ref{s-graded}, and the  invertible matrix $A$  from (\ref{A-gl}). Denote  their classical limit ($q \rightarrow 1$) by $A_0$,  $B_0$,
and $C_0$ ($=C$ as it  is independent of $q$). The subalgebra $\k$ is centralizing this classical limit.
As in the non-graded case, we describe it with the map $\theta=-w_\l\circ \tau$, which takes $\al\in \bar \Pi_\l$ to $ -\tilde\al\in \Rm^-_\g$.

The action of $\theta$ on the basis of $\h^*$ is  as follows (recall that $n$ stands for the rank of $\g$).
In the case of matrix $A_0$ it is
\begin{equation*}
\theta(\zeta_i)=\begin{cases}
\zeta_{i'},& \mathrm{if}\ i\leqslant m, \\
\zeta_{i}, &\mathrm{if}\ m<i,\\
\end{cases},
\quad \zeta_i= \delta_i\ \mathrm{or}\ \ve_i,
\end{equation*}
for $\g\l(N|2\m)$ and
\begin{equation*}\theta(\delta_i)=\begin{cases}
-\delta_i,&\mathrm{if}\ i\leqslant m, \\
\delta_{i}, &\mathrm{if}\ i>m,\\
\end{cases},
\quad\theta(\ve_i)=\ve_{i}.
\end{equation*}
for ortho-symplectic $\g$.
In the case of matrix $B_0$ for  ortho-symplectic $\g$   it is
\begin{equation*}
\theta(\delta_i)=\begin{cases}
-\delta_{i+1},& \mathrm{if}\ i<m\ \mathrm{is\ odd}, \\
-\delta_{i-1},& \mathrm{if}\ i\leqslant m\ \mathrm{is\ even}, \\
\delta_{i}, & \mathrm{if}\ i>m,\\
\end{cases}
\quad\theta(\ve_i)=\ve_{i}.
\end{equation*}
In the case of matrix $C_0$ for  ortho-symplectic $\g=\o\s\p(2|4\m)$  it is

\begin{equation*}
\theta(\delta_i)=\begin{cases}
-\delta_{i+1},& \mathrm{if}\ i<n\ \mathrm{is\ odd}, \\
-\delta_{i-1},& \mathrm{if}\ i\leqslant n-1\ \mathrm{is\ even}, \\
\end{cases}
\quad\theta(\ve_1)=-\ve_{1}.
\end{equation*}


The root basis of the subalgebra $\l$ is explicitly
$$
\Pi_\l=\begin{cases}
 \{\al_i\}_{i=m+1}^{n-m}, & \mbox{for } \g\l(N|2\m), \\
  \{\al_i\}_{i=m+1}^{n}, & \mbox{for }  \o\s\p(N|2\m),\ \mbox{and}\ \s\p\o(N|2\m) \ \mbox{related\ to}\ A_0, \\
\{\al_{2i+1} \}_{i=0}^{\frac{m}{2}-1}\cup  \{\al_i\}_{i=m+1}^{n}, & \mbox{for }  \o\s\p(N|2\m),\ \mbox{and}\ \s\p\o(N|2\m) \ \mbox{related\ to}\  B_0. \\
\{\al_{2i+1} \}_{i=0}^{\frac{n-3}{2}}, & \mbox{for }  \o\s\p(2|4\m), \ \mbox{related\ to}\  C_0. \\
\end{cases}
$$
For simple roots  $\al \in\bar \Pi_\l$ in the case of $A_0$  the roots $\tilde \al$ are given by
$$
\tilde\al_i=\al_{i'-1}, \quad i=1,\ldots, m-1,\quad i=m=\frac{N+2\m}{2},\quad  \tilde \al_i=\al_{i'-1},\quad i=m'+1,\ldots, n,
$$
$$
 \tilde \al_m=\sum_{l=m+1}^{n-m+1}\al_l, \quad \tilde \al_{n-m+1}=\sum_{l=m}^{n-m}\al_l,
$$
for $\g=\g\l(N|2\m)$,
$$ \tilde \al_i=\al_i, \quad i=1,\ldots, m-1,\quad  \tilde \al_m= \al_m+2\sum_{l=m+1}^{n}\al_l,$$
for $\g=\o\s\p(2\n+1|2\m)$,
$$ \tilde \al_i=\al_i, \quad i=1,\ldots, m-1,\quad  \tilde \al_m= \al_m+2\sum_{l=m+1}^{n-2}\al_l+\al_{n-1}+\al_{n},\quad m<n-1,$$
$$  \tilde \al_{n-1}= \al_n,\quad \tilde \al_{n}= \al_{n-1},\quad m=n-1,$$
for $\g=\o\s\p(2\n|2\m)$, and
$$ \tilde \al_i=\al_i, \quad i=1,\ldots, m-1,\quad  \tilde \al_m= \al_m+2\sum_{l=m+1}^{n-1}\al_l+\al_{n},\quad m<n-1, $$
$$\tilde \al_m=\al_{n-1}+\al_{n},\quad m=n-1,$$
for $\g=\s\p\o(2\n|2\m)$.

For simple roots  $\al \in\bar \Pi_\l$, in the case of $B_0$, the roots $\tilde \al$ are given by
$$ \tilde \al_{2i}=\sum_{l=2i-1}^{2i+1}\al_{l}, \quad i=1,\ldots, \frac{m}{2}-1,\quad  \tilde \al_m= \al_{m-1}+\al_{m}+2\sum_{l=m+1}^{n}\al_l,
$$
for $\g=\o\s\p(2\n+1|2\m)$,
$$ \tilde \al_{2i}=\sum_{l=2i-1}^{2i+1}\al_{l}, \quad i=1,\ldots, \frac{m}{2}-1,\quad  \tilde \al_m=\al_{m-1}+ \al_m+2\sum_{l=m+1}^{n-2}\al_l+\al_{n-1}+\al_{n},
$$
$$\tilde \al_{n-1}= \al_{n-2}+\al_n,\quad \tilde \al_{n}= \al_{n-2}+\al_{n-1},\quad m=n-1,$$
for $\g=\o\s\p(2\n|2\m)$,
$$ \tilde \al_{2i}=\sum_{l=2i-1}^{2i+1}\al_{l}, \quad i=1,\ldots, \frac{m}{2}-1,\quad  \tilde \al_m= \al_{m-1}+\al_{m}+2\sum_{l=m+1}^{n-1}\al_l+\al_{n},
$$
for $\g=\s\p\o(2\n|2\m)$.

For simple roots  $\al \in\bar \Pi_\l$, in the case of $C_0$, the roots $\tilde \al$ are given by
$$ \tilde \al_{2i}=\sum_{l=2i-1}^{2i+1}\al_{l}, \quad i=1,\ldots, \frac{n-3}{2},\
$$
$$\tilde \al_{n-1}= \al_{n-2}+\al_{n-1},\quad \tilde \al_{n}= \al_{n-2}+\al_{n},$$
for $\g=\o\s\p(2|4\m)$.

The Lie superalgebra $\k$ is generated by $\l$ and additional elements  $x_\al = e_\al+c_\al f_{\tilde \al}+ \grave c_{\al} u_\al$,   $h_{\tilde \al}-h_\al$, where $\al \in \bar\Pi_\l$,  and $c_\al\in \C^\times$,  $\grave c_{\al}\in \C$. The scalar $\grave c_\al\neq 0$ only if $\g=\g\l(N|2\m)$ and $\al=\al_{\frac{N+2\m}{2}}$; then $u_\al=h_\al$.
The mixture  parameters $c_\al$ are determined by the K-matrices.

\section{Coideal subalgebras and K-matrices}
\label{SecCoidSubA}
Recall that a total order on the set of positive roots  $ \Rm^+ \supset\Pi$ of $\g$ is called  normal if every   sum $\al +\beta\in \Rm^+$  with $\al, \beta \in \Rm^+$ is between $\al$ and $\beta$.

Choose a normal order $\bt^1,\bt^2,\ldots$ on $\Rm^+$ and extend $F_\bt=q^{h_\bt}f_\bt$ with simple $\bt$  to  all $\bt \in \Rm^+$ as it is done for $ e_{-\bt}$ in \cite{KT}.
Denote by $\Uc_m^-$ the subalgebra in $\Uc^-$ generated by $F_{\bt^i}$ with $i\leqslant m$.
\begin{lemma}
\label{coideal_norm_order}
For each   $m=1,\ldots,  |\Rm^+|$, $\Uc_m^-$ is a left coideal subalgebra in $\U_q(\g)$.
Furthermore,
$$
\Delta(F_{\bt^m})=  (F_{\bt^m}\tp 1+q^{h_{\bt^m}}\tp F_{\bt^m}) +  U_q(\g)\tp \Uc^-_{m-1}.
$$
\end{lemma}

\begin{proof}
This readily follows from  \cite{KT}, Prop. 8.3. upon the assignment $F_\al\mapsto e_{-\al}$ on simple root vectors, extended as an algebra automorphism
to $U_q(\b_-)$ (and a coalgebra anti-isomorphism).
\end{proof}
We  apply this fact to quantize $U(\k)$.
Pick  $\al\in \bar \Pi_\l$, put  $\al'=\tau(\al)\in \bar \Pi_\l$ and $\tilde \al=w_\l(\al')\in \Rm^+$. Consider the root system generated by $\al'$ and $\Pi_\l$;
denote by $\Pi_{\k_{\tilde \al}}$  its connected component containing $\al'$, and by $\k_{\tilde \al}$ the subalgebra generated by the corresponding
simple root vectors. Denote also $\Pi_{\l_{\tilde \al}}=\Pi_{\k_{\tilde \al}} \cap \Pi_\l$ and by $\l_{\tilde \al}$ the corresponding subalgebra in $\l$.
\begin{propn}
  For each $\al\in \Pi_\al$ there is a normal order on $\Rm^+_{\k_{\tilde \al}}$ such that
\begin{itemize}
  \item $\al'$ is in the rightmost position and all $\Rm^+_{\l_{\tilde \al}}$ are on the left.
  \item the root $\tilde \al$ is next to the right after $\Rm^+_{\l_{\tilde \al}}$.
\end{itemize}
\end{propn}
\begin{proof}
  In the non-graded case this is derived  from the properties of the longest element of the Weyl group.
  In our case this is checked by a direct examination.
\end{proof}
\noindent
We can conclude now, by   Lemma \ref{coideal_norm_order},
that for each $\al\in \bar \Pi_\l$ the elements $F_{\tilde \al}$ and $F_\mu$ with  $\mu\in \Pi_{\l_{\tilde\al}}$
form a left   coideal subalgebra in $U_q(\g)$.


\begin{thm} \label{thm}
  The subalgebra $U_q(\k)\subset U_q(\g)$ generated by $e_\al,f_\al, q^{\pm h_\al}$ with  $\al \in \Pi_\l$, and by $X_\al=q^{h_{\tilde \al}-h_\al}e_\al +c_\al F_{\tilde\al} +\grave c_{\al}(q^{u_\al}-1)$,
  $q^{\pm(h_{\tilde \al}-h_\al)}$ with $\al \in \bar\Pi_\l$, $u_\al\in \c$,  is a left coideal   in $U_q(\g)$.
\end{thm}
\begin{proof}
 First set $\grave c_{\al}=0$.
   For each   $\al \in \bar\Pi_\l $, we find, using Lemma \ref{coideal_norm_order}:
$$
\Delta(F_{\tilde \al})\in  (q^{h_{\tilde\al}}\tp F_{\tilde\al}+F_{\tilde \al}\tp 1) +  U_q(\g)\tp U_q(\l_{\tilde \al}),
$$
$$
\Delta (q^{h_{\tilde \al}-h_\al}e_\al)= q^{h_{\tilde \al}}\tp q^{h_{\tilde \al}-h_\al}e_\al+q^{h_{\tilde \al}-h_\al}e_\al \tp q^{h_{\tilde \al}-h_\al}.
$$
Adding the lines together we arrive at
$$
\Delta(X_{\tilde \al})\in q^{h_{\tilde\al}} \tp X_{\al}+ F_{\tilde\al}\tp 1  + q^{h_{\tilde \al}-h_\al}e_\al\tp  q^{h_\al -h_{\tilde \al}}  +   U_q(\g)\tp  U_q(\l_{\tilde \al}),
$$
as required.
It is straightforward to see that the coproduct of the  term with $\grave c_{\al}\neq0$ is in $U_q(\g)\otimes U_q(\k)$ too. This completes the proof.
\end{proof}

In the next sections we relate the RE matrices presented in Theorem \ref{s-graded} with  coideal subalgebras.
We require that such a matrix   commutes with $\pi\bigl(U_q(\l)\bigr)$ and all $\pi(q^{h_{\tilde \al}-h_\al})$, where $\al \in \bar \Pi_\l$.
Th
e values of $c_\al$  and $\grave c_{\al}$ are determined from this requirement.
The constant $\grave c_\al$ is distinct from zero only if  $\g=\g\l(N|2\m)$ with $\al=\tilde\al_m=\al_m,\ m=\frac{N+2\m}{2}$, and $\mu\neq -\la$.
Expressions for $F_{\tilde \al}$ with non-simple $\tilde \al$  will be provided explicitly in terms of q-commutator defined
 as  $$[F_{\al}, F_{\beta}]_{q^{\pm i}}=F_{\al}F_{\beta}-(-1)^{(|F_{\al}||F_{\beta}|)}q^{\mp\lceil \frac{i}{2}\rceil (\alpha,\beta)}F_{\beta} F_{\al},\quad \forall\alpha,\beta\in \Rm^+.$$
Here $i=0,1,2$ and $i\mapsto \lceil \frac{i}{2}\rceil$ sends it to $0,1,1$.

\subsection{K-matrices of type $A$}
It is found that $\tilde \al_i=\al_i$ for all $i<m$, and $
c_{\al_{i}}=
-\frac{y_{i+1}}{q y_{i}}
$, for  ortho-symplectic $\g$. The four types of $\g$ will be further treated separately.

$\bullet\quad \g=\g\l(N|2\m)$\\
The Satake  diagrams fall into two classes:
\begin{center}
\begin{picture}(160,80)
\put(9,8.5){\framebox(3,3)}
\put(49,8.5){\framebox(3,3)}
\put(100,8.5){\framebox(3,3)}
\put(140,8.5){\textcolor{black}{\rule{3pt}{3pt}}}
\put(12,10){\line(1,0){35}}
\put(52,10){\line(1,0){10}}
\put(88,10){\line(1,0){10}}
\put(67,10){$\ldots$}
\put(103,10){\line(1,0){36}}

\put(9,68.5){\framebox(3,3)}
\put(49,68.5){\framebox(3,3)}
\put(100,68.5){\framebox(3,3)}
\put(140,68.5){\textcolor{black}{\rule{3pt}{3pt}}}
\put(12,70){\line(1,0){35}}
\put(52,70){\line(1,0){10}}
\put(88,70){\line(1,0){10}}
\put(67,70){$\ldots$}
\put(103,70){\line(1,0){36}}

\put(141,12){\line(0,1){18}}
\put(140,30){\textcolor{black}{\rule{3pt}{3pt}}}
\put(140,37){$\vdots$}
\put(140,50){\textcolor{black}{\rule{3pt}{3pt}}}
\put(141,52){\line(0,1){17}}

 \qbezier(8,16)(0,40)(8 ,64) \put(6.9,19.5){\vector(1,-3){2}} \put(7.0,61.3){\vector(1,3){2}}
 \qbezier(48,16)(40,40)(48 ,64) \put(46.9,19.5){\vector(1,-3){2}} \put(47.0,61.3){\vector(1,3){2}}
\qbezier(98,16)(90,40)(98 ,64) \put(96.9,19.5){\vector(1,-3){2}} \put(97.0,61.3){\vector(1,3){2}}

\put(12,74){$\al_1$}
\put(102,74){$\al_m$}

\put(12,3){$\al_n$}
\put(102,3){$\al_{ N+2\m-m}$}

 \end{picture}
\quad \quad
\begin{picture}(160,85)
\put(65,80){$m=\frac{N+2\m}{2}$}

\put(9,8.5){\framebox(3,3)}
\put(49,8.5){\framebox(3,3)}
\put(100,8.5){\framebox(3,3)}
\put(140,8.5){\framebox(3,3)}
\put(12,10){\line(1,0){35}}
\put(52,10){\line(1,0){10}}
\put(88,10){\line(1,0){10}}
\put(67,10){$\ldots$}
\put(103,10){\line(1,0){36}}

\put(9,68.5){\framebox(3,3)}
\put(49,68.5){\framebox(3,3)}
\put(100,68.5){\framebox(3,3)}
\put(140,68.5){\framebox(3,3)}
\put(12,70){\line(1,0){35}}
\put(52,70){\line(1,0){10}}
\put(88,70){\line(1,0){10}}
\put(67,70){$\ldots$}
\put(103,70){\line(1,0){36}}

\put(141.5,11.5){\line(1,1){27}}
 \put(170,40){\circle{3}}
\put(141.5,68.5){\line(1,-1){27}}

 \qbezier(8,16)(0,40)(8 ,64) \put(6.9,19.5){\vector(1,-3){2}} \put(7.0,61.3){\vector(1,3){2}}
 \qbezier(48,16)(40,40)(48 ,64) \put(46.9,19.5){\vector(1,-3){2}} \put(47.0,61.3){\vector(1,3){2}}
\qbezier(98,16)(90,40)(98 ,64) \put(96.9,19.5){\vector(1,-3){2}} \put(97.0,61.3){\vector(1,3){2}}
\qbezier(138,16)(130,40)(138 ,64) \put(136.9,19.5){\vector(1,-3){2}} \put(137.0,61.3){\vector(1,3){2}}

\put(12,73){$\al_1$}
\put(172,42){$\al_m$}

\put(12,3){$\al_n$}
 \end{picture}

\end{center}
with the mixture parameters
$$
c_{\al_{i}}=(-1)^{\delta_i^{\m'}}
\frac{y_{i+1}}{y_{i}},\quad \mathrm{if}\quad i< m \quad \mathrm{or} \quad i> N+2\m-m.
$$
The roots $\tilde\alpha_m$ and $\tilde\alpha_{N+2\m-m}$ are not simple when $m<\frac{N+2\m}{2}$.  The corresponding root vectors can be defined as
$$F_{\tilde \al_m}=[\dots[F_{m+1}, F_{m+2}]_{ q} ,\dots  F_{n-m+1}]_{q}, \quad F_{\tilde \al_{n-m+1}}=[\dots[F_{m}, F_{m+1}]_{\bar q} ,\dots  F_{n-m}]_{\bar q}.
$$
The mixture parameters are expressed by
$$
c_{\al_{m}}=  \frac{(-1)^{N} \mu}{y_m}, \quad
c_{\al_{N+2\m-m}}=\begin{cases}
              \frac{(-1)^{N+1+\delta_m^{\m}} q^{2(N+2\m) - 4m - 3} \lambda }{y_m}, & \mbox{if }\ \m\leq m, \\
              \frac{(-1)^{N+1} q^{2(N-2\m) + 4m + 3}\lambda }{y_m} , & \mbox{if }\ m<\m.
            \end{cases}
$$
The case $m=\frac{N+2\m}{2}$ with $\m<m$ is described by the diagram on the right. The simple root $\al_m=\al$ corresponds to the following  mixed  vector

$$X_\al=e_{\al} +c_{\al} F_{\al} + \grave c_{\al} (q^{h_{\al}}-1),$$
where
$$
    c_{\al}=\frac{-\lambda\mu q}{y_m^2}, \quad
    \grave c_{\al}=\frac{(\mu+\lambda) q}{(q^{2} - 1)y_m}.
$$
Remark that $\l=\{0\}$, and $\c=\h$ in this case. The element $u_\al\in \c$ is taken equal to $h_\al$.

$\bullet\quad\g=\o\s\p(2\n+1|2\m)$\\
There are two Satake diagrams leading the K-matries of type $A$:
\be
\begin{picture}(200,20)

\put(2,10){\line(1,0){10}}
\put(38,10){\line(1,0){10}}
\put(0,10){\circle{3}}
\put(17,10){$\ldots$}
\put(50,10){\circle{3}}
\put(51.5,10){\line(1,0){27}}
\put(79,8.5){\framebox(3,3)}
\put(82,10){\line(1,0){27}}

\put(112,10){\line(1,0){10}}
\put(148,10){\line(1,0){10}}
\put(109,8.5){\textcolor{black}{\rule{3pt}{3pt}}}
\put(127,10){$\ldots$}
\put(158,8.5){\textcolor{black}{\rule{3pt}{3pt}}}
\put(161,8.5){\line(1,0){24.5}}
\put(161,11.5){\line(1,0){24.5}}
\put(182,7){$>$}
\put(189,9){\textcolor{black}{\rule{3pt}{3pt}}}

\put(82,14){$\al_m$}

\put(0,14){$\al_1$}

\put(187,14){$\al_{n}$}

 \end{picture}
 \quad\quad
\begin{picture}(120,20)
\put(0,10){\circle{3}}
\put(1.5,10){\line(1,0){27}}
\put(32,10){\line(1,0){10}}
\put(68,10){\line(1,0){10}}
\put(30,10){\circle{3}}
\put(80,10){\circle{3}}
\put(47,10){$\ldots$}

\put(109,9){\framebox(3,3)}

\put(81,8.5){\line(1,0){24.5}}
\put(81,11.5){\line(1,0){24.5}}
\put(101.5,7){$>$}

\put(107,14){$\al_m$}

\put(0,14){$\al_1$}

 \end{picture}
 \nn
\ee
Note that the diagram on the right is admissible. It does not fall under conditions of Lemma \ref{isolated odd} because the odd root $\al_m$ is not grey.
 The composite root vectors for the left diagram are
$$F_{\tilde \al_m}=[\dots[[[\dots[[F_{m}, F_{m+1}]_{q}, F_{m+2}]_{q},\dots  F_{n}]_{q},F_{n}],F_{n-1}]_{q},\dots F_{m+1}]_{q},\ m< n-1,$$
$$F_{\tilde \al_m}=[[F_{n-1}, F_{n}]_{q},F_{n}],\quad  m=n-1,$$
with mixture coefficient
$$
c_{\al_{m}}=\frac{(-1)^{(n- m-\delta_m^\m)} \bar q^{2\delta_m^{\m}}\lambda}{y_m }.
$$
In the scenario with $\m=m=n$ relative to the second diagram, the mixture constant is
$$c_{\al_m}=-\frac{\lambda}{y_m q}.$$

$\bullet\quad\g=\o\s\p(2\n|2\m)$\\
The case $m\leqslant n-2$ includes two diagrams
 
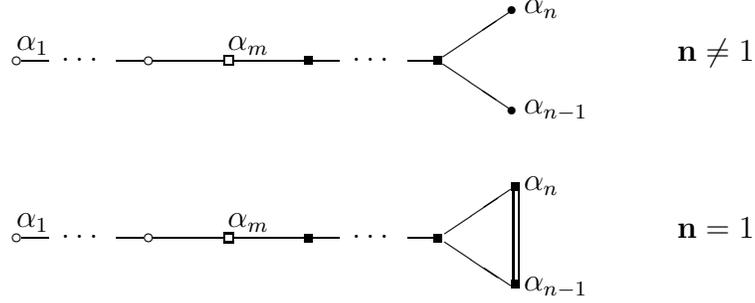
\begin{center}
\begin{picture}(200,40)

\put(2,20){\line(1,0){10}}
\put(38,20){\line(1,0){10}}
\put(0,20){\circle{3}}
\put(17,20){$\ldots$}
\put(50,20){\circle{3}}
\put(51.5,20){\line(1,0){27}}
\put(79,18.5){\framebox(3,3)}
\put(82,20){\line(1,0){27}}

\put(112,20){\line(1,0){10}}
\put(148,20){\line(1,0){10}}
\put(109,18.5){\textcolor{black}{\rule{3pt}{3pt}}}
\put(127,20){$\ldots$}
\put(158,18.5){\textcolor{black}{\rule{3pt}{3pt}}}
\put(161.5,21.5){\line(3,2){25}}
\put(161.5,18.5){\line(3,-2){25}}
\put(187.5,39){\circle*{3}}
\put(187.5,1){\circle*{3}}

\put(0,24){$\al_1$}
\put(80,24){$\al_m$}

\put(192,0){$\al_{n-1}$}
\put(192,38){$\al_{n}$}
\put(250,20){$\n\neq 1$}
 \end{picture}
 \end{center}

\vspace{0.5cm}

\begin{center}
\begin{picture}(200,40)

\put(2,20){\line(1,0){10}}
\put(38,20){\line(1,0){10}}
\put(0,20){\circle{3}}
\put(17,20){$\ldots$}
\put(50,20){\circle{3}}
\put(51.5,20){\line(1,0){27}}
\put(79,18.5){\framebox(3,3)}
\put(82,20){\line(1,0){27}}

\put(112,20){\line(1,0){10}}
\put(148,20){\line(1,0){10}}
\put(109,18.5){\textcolor{black}{\rule{3pt}{3pt}}}
\put(127,20){$\ldots$}
\put(158,18.5){\textcolor{black}{\rule{3pt}{3pt}}}
\put(162,21.5){\line(3,2){25}}
\put(162,18.5){\line(3,-2){25}}
\put(187.5,38){\textcolor{black}{\rule{3pt}{3pt}}}
\put(187.5,1){\textcolor{black}{\rule{3pt}{3pt}}}

\put(0,24){$\al_1$}
\put(80,24){$\al_m$}

\put(192,0){$\al_{n-1}$}
\put(192,38){$\al_{n}$}

\put(190,3){\line(0,1){35}}
\put(188,3){\line(0,1){35}}
\put(250,20){$\n= 1$}
 \end{picture}
\end{center}
 The root vectors are
$$
F_{\tilde \al_m}=[\dots[[\dots[[F_{m}, F_{m+1}]_{q}, F_{m+2}]_{q},\dots F_{n}]_{q},F_{n-2}]_{q},\dots F_{m+1}]_{q},
\quad \mbox{for}\quad m<n-2,
$$
$$F_{\tilde \al_m}=[[F_{n-2}, F_{n-1}]_{q}, F_{n}]_{q}\quad \mbox{for}\quad  m=n-2.
$$
In both cases, the parameters are
$$
c_{\al_{m}}=\frac{(-1)^{(n- m+1+\delta_m^{\m})}\bar q^{2\delta_m^\m}\lambda}{y_m }.
$$
The case $m=n-1$ corresponds to the diagram
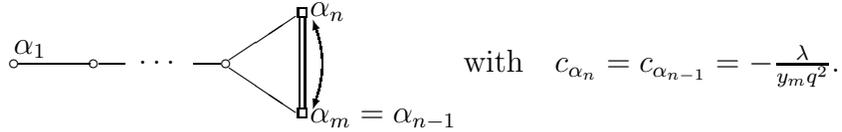
\begin{center}
 \begin{picture}(300,50)

\put(0,30){\circle{3}}
\put(1.5,30){\line(1,0){27}}

\put(32,30){\line(1,0){10}}
\put(68,30){\line(1,0){10}}
\put(30,30){\circle{3}}
\put(47,30){$\ldots$}
\put(80,30){\circle{3}}
\put(81.5,31.5){\line(3,2){25}}
\put(81.5,28.5){\line(3,-2){25}}
\put(107.5,48){\framebox(3,3)}
\put(107.5,10){\framebox(3,3)}
\qbezier(114,16)(119,30)(114 ,44) \put(115,19){\vector(-1,-3){2}} \put(115,41){\vector(-1,3){2}}

\put(0,34){$\al_1$}
\put(112,8){$\al_m=\al_{n-1}$}
\put(112,48){$\al_n$}

\put(110,13){\line(0,1){35}}
\put(108,13){\line(0,1){35}}

\put(170,27){$\mbox{with} \quad
c_{\al_{n}}=c_{\al_{n-1}}=-\frac{\lambda}{y_{m}q^2}$.}
 \end{picture}
\end{center}

$\bullet\quad\g=\s\p\o(2\n|2\m)$\\
The graded Satake diagram is
 \begin{center}
\begin{picture}(230,30)

\put(0,10){\circle{3}}
\put(1.5,10){\line(1,0){27}}
\put(30,10){\circle{3}}
\put(31.5,10){\line(1,0){27}}
\put(60,10){\circle{3}}

\put(62,10){\line(1,0){10}}
\put(77,10){$\ldots$}
\put(98,10){\line(1,0){10}}

\put(108,8.5){\framebox(3,3)}
\put(111,10){\line(1,0){28}}
\put(139,8.5){\textcolor{black}{\rule{3pt}{3pt}}}

\put(142,10){\line(1,0){10}}
\put(157,10){$\ldots$}
\put(178,10){\line(1,0){10}}

\put(188,8.5){\textcolor{black}{\rule{3pt}{3pt}}}
\put(194,8.5){\line(1,0){24.5}}
\put(194,11.5){\line(1,0){24.5}}
\put(190,7){$<$}
\put(219,10){\circle*{3}}

\put(0,14){$\al_1$}
\put(107,14){$\al_m$}
\put(217,14){$\al_{n}$}

 \end{picture}
\end{center}
with coefficients
$$ c_{\al_{m}}=(-1)^{n -m+\delta_{m}^\m}\frac{\bar q^{2\delta_{m}^\m}\lambda}{ y_m},$$
and root vectors
$$F_{\tilde \al_m}=[\dots[F_{m}, F_{m+1}]_{ q}, F_{m+2}]_{ q},\dots F_{n}]_{ q^{2}},F_{n-1}]_{ q},\dots F_{m+1}]_{ q},\quad m<n-1,$$
$$F_{\tilde \al_m}=[F_{n-1},F_{n}]_{q^{2}},\quad m=n-1.$$
This completes the description of the diagrams for this type of K-matrix.

\subsection{K-matrices of type $B$}

For all cases, we find
$$F_{\tilde \al_{2i}}=[[F_{2i}, F_{2i+1}]_{q}, F_{2i-1}]_{q}\ \mathrm{with}\ i = 1,\ldots, \frac{m}{2}-1, \quad \mbox{and }$$
$$
c_{\al_{2i}}=
-\frac{z_{2i+1}}{q z_{2i-1}},\quad \mathrm{if}\  2i<m.
$$

$\bullet\quad\g=\o\s\p(2\n+1|2\m)$\\
The Satake diagram is
 \begin{center}
\begin{picture}(230,30)

\put(0,10){\circle*{3}}
\put(1.5,10){\line(1,0){27}}
\put(30,10){\circle{3}}
\put(31.5,10){\line(1,0){27}}
\put(60,10){\circle*{3}}

\put(62,10){\line(1,0){10}}
\put(77,10){$\ldots$}
\put(98,10){\line(1,0){10}}

\put(108,8.5){\framebox(3,3)}
\put(111.5,10){\line(1,0){27}}
\put(140,8.5){\textcolor{black}{\rule{3pt}{3pt}}}

\put(142,10){\line(1,0){10}}
\put(157,10){$\ldots$}
\put(178,10){\line(1,0){10}}

\put(188,8.5){\textcolor{black}{\rule{3pt}{3pt}}}
\put(191,8.5){\line(1,0){23.5}}
\put(191,11.5){\line(1,0){23.5}}
\put(211,7){$>$}
\put(218,8.5){\textcolor{black}{\rule{3pt}{3pt}}}

\put(0,14){$\al_1$}
\put(107,14){$\al_m$}
\put(217,14){$\al_{n}$}

 \end{picture}
\end{center}
For even $m=2i$, the non-simple root vectors are
\begin{equation*}
   F_{\tilde \al_m}=[\dots[[[\dots[[F_{m}, F_{m+1}]_{q}, F_{m+2}]_{q},\dots  F_{n}]_{q},F_{n}],F_{n-1}]_{q},\dots F_{m+1}]_{q},F_{m-1}]_{q},\quad m<n,
\end{equation*}
with
$$
  c_{\al_{m}}=(-1)^{n+\delta_{m}^\m}\frac{q^{-2\delta_{m}^\m-1}\lambda}{z_{m-1}},
$$
and
 $$F_{\tilde \al_m}=[F_{n-1},F_{n}]_{ q^{-1}},\quad m=n, $$
with
$$
  c_{\al_{m}}=\frac{\lambda }{z_{n-1}q^{3}}.
$$

$\bullet\quad\g=\o\s\p(2\n|2\m)$\\
For $m<n-1$, and $\n\neq 1$, we have the diagram
  
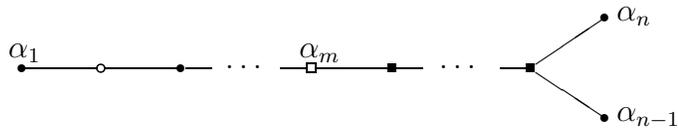
\begin{center}
\begin{picture}(220,40)

 \put(20,20){\circle*{3}}
\put(21.5,20){\line(1,0){27}}
\put(50,20){\circle{3}}
\put(51.5,20){\line(1,0){27}}

\put(82,20){\line(1,0){10}}
\put(118,20){\line(1,0){10}}
\put(80,20){\circle*{3}}
\put(97,20){$\ldots$}

\put(50,20){\circle{3}}
\put(51.5,20){\line(1,0){27}}

\put(128,18.5){\framebox(3,3)}
\put(131.5,20){\line(1,0){27}}

\put(158.5,18.5){\textcolor{black}{\rule{3pt}{3pt}}}

\put(211,18.5){\textcolor{black}{\rule{3pt}{3pt}}}
\put(214.5,21.5){\line(3,2){25}}
\put(214.5,18.5){\line(3,-2){25}}
\put(240.5,39){\circle*{3}}
\put(240.5,1){\circle*{3}}

\put(15,24){$\al_1$}
\put(125,24){$\al_m$}
\put(245,0){$\al_{n-1}$}
\put(245,38){$\al_{n}$}

\put(161.5,20){\line(1,0){10}}
\put(201,20){\line(1,0){10}}
\put(178,20){$\ldots$}

 \end{picture}
\end{center}
 When $m<n-1$, and $\n=1$, there is another diagram
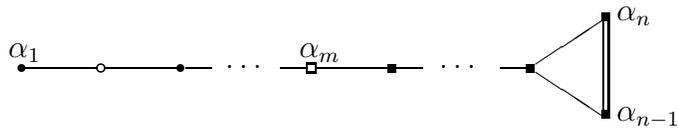
 \begin{center}
\begin{picture}(220,40)

 \put(20,20){\circle*{3}}
\put(21.5,20){\line(1,0){27}}
\put(50,20){\circle{3}}
\put(51.5,20){\line(1,0){27}}

\put(82,20){\line(1,0){10}}
\put(118,20){\line(1,0){10}}
\put(80,20){\circle*{3}}
\put(97,20){$\ldots$}

\put(50,20){\circle{3}}
\put(51.5,20){\line(1,0){27}}

\put(128,18.5){\framebox(3,3)}
\put(131.5,20){\line(1,0){27}}

\put(158.5,18.5){\textcolor{black}{\rule{3pt}{3pt}}}

\put(211,18.5){\textcolor{black}{\rule{3pt}{3pt}}}
\put(214.5,21.5){\line(3,2){25}}
\put(214.5,18.5){\line(3,-2){25}}
\put(239.5,38){\textcolor{black}{\rule{3pt}{3pt}}}
\put(239.5,1){\textcolor{black}{\rule{3pt}{3pt}}}

\put(15,24){$\al_1$}
\put(125,24){$\al_m$}
\put(245,0){$\al_{n-1}$}
\put(245,38){$\al_{n}$}

\put(161.5,20){\line(1,0){10}}
\put(201,20){\line(1,0){10}}
\put(178,20){$\ldots$}

\put(242,3){\line(0,1){35}}
\put(240,3){\line(0,1){35}}
 \end{picture}
\end{center}
The mixture parameter is
\begin{equation*}
c_{\al_{m}}=
(-1)^{n+1+\delta_{m}^\m}\frac{q^{-2\delta_{m}^\m-1}\lambda}{z_{m-1}},\quad m< n-1
\end{equation*}
and root vector
   $$F_{\tilde \al_m}=[[\dots[[\dots[[F_{m}, F_{m+1}]_{q}, F_{m+2}]_{q},\dots  F_{n}]_{q},F_{n-2}]_{q},\dots F_{m+1}]_{q},F_{m-1}]_{q}.$$
In the case of $m=n-1$, it corresponds to the diagram
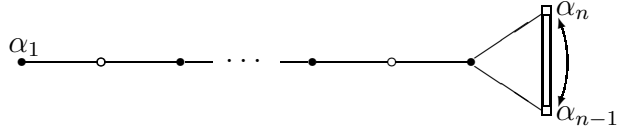
\begin{center}
\begin{picture}(220,50)

 \put(20,20){\circle*{3}}
\put(21.5,20){\line(1,0){27}}
\put(50,20){\circle{3}}
\put(51.5,20){\line(1,0){27}}

\put(82,20){\line(1,0){10}}
\put(118,20){\line(1,0){10}}
\put(80,20){\circle*{3}}
\put(97,20){$\ldots$}

\put(50,20){\circle{3}}
\put(51.5,20){\line(1,0){27}}

\put(130,20){\circle*{3}}
\put(131.5,20){\line(1,0){27}}

\put(160,20){\circle{3}}
\put(161.5,20){\line(1,0){27}}

\put(190,20){\circle*{3}}
\put(191.5,21.5){\line(3,2){25}}
\put(191.5,18.5){\line(3,-2){25}}
\put(217,38){\framebox(3,3)}
\put(217,0){\framebox(3,3)}

\put(15,24){$\al_1$}

\put(222,-2){$\al_{n-1}$}
\put(222,38){$\al_n$}

\put(220,3){\line(0,1){34.5}}
\put(217,3){\line(0,1){34.5}}

\qbezier(224,6)(229,20)(224 ,34) \put(225,9){\vector(-1,-3){2}} \put(225,31){\vector(-1,3){2}}
 \end{picture}
\end{center}
Composite root vectors are  defined as
$$F_{\tilde \al_{n-1}}=[F_{n},F_{n-2}]_{q},\quad
F_{\tilde \al_n}=[F_{n-1},F_{n-2}]_{ q},\quad
\mbox{and} \quad
c_{\al_{n-1}}=c_{\al_{n}}=-
\frac{\lambda}{q^{3}z_{n-1}}.
$$

$\bullet\quad\g=\s\p\o(2\n|2\m)$\\
The  diagram is
 \begin{center}
\begin{picture}(230,30)

\put(0,10){\circle*{3}}
\put(1.5,10){\line(1,0){27}}
\put(30,10){\circle{3}}
\put(31.5,10){\line(1,0){27}}
\put(60,10){\circle*{3}}

\put(62,10){\line(1,0){10}}
\put(77,10){$\ldots$}
\put(98,10){\line(1,0){10}}

\put(108,8.5){\framebox(3,3)}
\put(111.5,10){\line(1,0){28}}
\put(139,8.5){\textcolor{black}{\rule{3pt}{3pt}}}

\put(142,10){\line(1,0){10}}
\put(157,10){$\ldots$}
\put(178,10){\line(1,0){10}}

\put(188,8.5){\textcolor{black}{\rule{3pt}{3pt}}}
\put(194,8.5){\line(1,0){24.5}}
\put(194,11.5){\line(1,0){24.5}}
\put(190,7){$<$}
\put(219,10){\circle*{3}}

\put(0,14){$\al_1$}
\put(107,14){$\al_m$}
\put(217,14){$\al_{n}$}

 \end{picture}
\end{center}
Depending on the value of $m$ we define the root vectors as
$$F_{\tilde \al_m}=[[\dots[[\dots[[F_{m}, F_{m+1}]_{q}, F_{m+2}]_{q},\dots F_{n}]_{q^{2}},F_{n-1}]_{q},\dots F_{m+1}]_{q},F_{m-1}]_{q},\quad \mathrm{if}\ m<n-1,$$
$$F_{\tilde \al_{m}}=[[F_{n-1}, F_{n}]_{q^2}, F_{n-2}]_{q},\quad \mathrm{if}\ m=n-1\in 2\Z,$$
with the mixture parameter
\begin{equation*}
c_{\al_{m}}=
(-1)^{n+\delta_{m}^\m}\frac{q^{-2\delta_{m}^\m-1}\lambda}{z_{m-1}},\quad m\leq\m.
\end{equation*}
This completes our description of type I coideal subalgebras and their K-matrices.
 \subsection{K-matrix of type $C$}
\label{SecKmatC}
The matrix $C$ is  related with  the algebra $\g=\o\s\p(2|4\m)$ and its  Satake diagram
\begin{center}
\begin{picture}(220,50)

 \put(20,20){\circle*{3}}
\put(21.5,20){\line(1,0){27}}
\put(50,20){\circle{3}}
\put(51.5,20){\line(1,0){27}}

\put(82,20){\line(1,0){10}}
\put(118,20){\line(1,0){10}}
\put(80,20){\circle*{3}}
\put(97,20){$\ldots$}

\put(50,20){\circle{3}}
\put(51.5,20){\line(1,0){27}}

\put(130,20){\circle*{3}}
\put(131.5,20){\line(1,0){27}}

\put(160,20){\circle{3}}
\put(161.5,20){\line(1,0){27}}

\put(190,20){\circle*{3}}
\put(191.5,21.5){\line(3,2){25}}
\put(191.5,18.5){\line(3,-2){25}}
\put(217,38){\framebox(3,3)}
\put(217,0){\framebox(3,3)}

\put(15,24){$\al_1$}

\put(222,-2){$\al_{n-1}$}
\put(222,38){$\al_n$}

\put(220,3){\line(0,1){34.5}}
\put(217,3){\line(0,1){34.5}}

 \end{picture}
\end{center}
The composite root vectors of roots $\tilde \al$, where $\al \in \bar \Pi_\l$, are
$$F_{\tilde \al_{2i}}=[[F_{2i}, F_{2i+1}]_{q}, F_{2i-1}]_{q}\ \mathrm{with}\ c_{\al_{2i}}=
-\frac{x_{2i+1}}{q x_{2i-1}},\ \mathrm{where}\ i = 1,\ldots, \frac{n-3}{2},$$
for even $\al$.
Also, there are two odd root vectors
 $$F_{\tilde \al_n}=[F_{n},F_{n-2}]_{ q},\quad F_{\tilde \al_{n-1}}=[F_{n-1},F_{n-2}]_{ q}, $$
with mixture parameters
$$
  c_{\al_{n}}=-\frac{x_{n+2} }{x_{n}q}, \quad  c_{\al_{n-1}}=-\frac{x_n }{x_{n-2}q}.
$$

\vspace{20pt}

\noindent
\underline{\large \bf Acknowledgement}

\vspace{10pt}
\noindent
This work is done at the Center of Pure Mathematics MIPT.
It is financially supported  by Russian Science Foundation  grant 23-21-00282.

The first author (D. Algethami) is   thankful to the Deanship of Graduate Studies and Scientific Research at University of Bisha for the financial support through the Scholars Research Support Program of the University.

A. Mudrov and V. Stukopin are grateful to V. Zhgoon for valuable  discussions of spherical manifolds.

The authors are profoundly indebted to the anonymous reviewer for careful reading of the manuscript and for extremely valuable comments
that served to  improvement of its first version.

\vspace{15pt}

\underline{\large \bf Data Availability.}

\vspace{10pt}

 Data sharing not applicable to this article as no datasets were generated or analysed during the current study.

\subsection*{Declarations}

\underline{\large \bf Competing interests.}

\vspace{10pt}

The authors have no competing interests to declare that are relevant to the content of this article.


\begin{thebibliography}{A}
\bibitem{Let} Letzter, G.: {\em Symmetric pairs for quantized enveloping algebras}, J. Algebra, \#2, {\bf 220} (1999),  729--767.

\bibitem{K} Kolb, S.:{\em Quantum symmetric Kac–Moody pairs}, Adv. Math., {\bf 267} (2014), 395--469.


\bibitem{BK} Balagovi$\acute{\mathrm c}$, M., Kolb, S.:{\em Universal K-matrix for quantum symmetric pairs}, J. Reine Angew. Math., {\bf 747} (2019), 299--353.

\bibitem{RV}  Regelskis, V., Vlaar, B.:{\em Quasitriangular coideal subalgebras of $U_q(\g)$ in terms of generalized Satake diagrasms}, Bull. LMS, {\bf 52} \#4 (2020), 561--776.

\bibitem{AV} Appel, A. and Vlaar, B.: {\em Universal K-matrices for quantum Kac-Moody algebras}. Representation
Theory of the AMS {\bf 26}, no. 26 (2022), 764--824.



\bibitem{Ma} Manin, Y.I.: (1997). Introduction to Supergeometry. In: Gauge Field Theory and Complex Geometry. Grundlehren der mathematischen Wissenschaften, vol 289. Springer, Berlin, Heidelberg.

\bibitem{VMP} A. A. Voronov, Yu. I. Manin, I. B. Penkov :  {\em Elements of supergeometry}, J. Soviet Math., {\bf 51} \#1 (1990), 2069--2083.



\bibitem{Dir} Dirac, P.: (1967). Principles of Quantum Mechanics (4th ed.). London: Oxford University Press.

\bibitem{FP} Faddeev, L. D. Popov, V.: {\em Feynman diagrams for the Yang-Mills field}, Phys. Lett. B., {\bf 25} \# 1 (1967) 29--30,

\bibitem{Sh} Sherman, A. : {\em Spherical supervarieties}, Ann. de l'institut Fourier,  {\bf 71} (2021), 4,  1449 -- 1492.

\bibitem{Sh1} Sherman, A. : {\em Spherical indecomposable representations of Lie superalgebras}, J. Algebra,  {\bf 547} (2020),   262 -- 311.

\bibitem{Y1} Yamane, H.: {\em Quantized Enveloping Algebras Associated with Simple Lie Superalgebras and Their Universal R-matrices},  Publ. Res. Inst. Math. Sci. {\bf 30} \#1 (1994),  15--87.

\bibitem{Shen}  Shen, Y.: {\em Quantum supersymmetric pairs and $\iota$Schur duality of type AIII}, arXiv:2210.01233.

\bibitem{KY} Kolb, S., Yakimov, M.:{\em  Symmetric pairs for Nichols algebras of diagonal type via star products}, Adv.
Math. {\bf 365} (2020), 107042.
 \bibitem{ShenWang} Y. Shen, W. Wang: {\em Quantum supersymmetric pairs of basic types}, arXiv:2408.02874.

\bibitem{Y2} Yamane, H.: {\em Quantized Enveloping Algebras Associated with Simple Lie Superalgebras
and Their Universal R-matrices},
Publ. RIMS, Kyoto Univ.
{\bf 30} (1994), 15--87.

\bibitem{VK} Vinberg, E. B. Kimelfeld: {\em Homogeneous Domains on Flag Manifolds
and Spherical Subgroups}, Func. Anal. Appl., {\bf 12} (1978), 168--174.

\bibitem{AlM}	Algethami, D.,   Mudrov, A.: {\em Quantum symmetric conjugacy classes of non-exceptional groups}, J. Math. Phys., {\bf 64} (2023), 081703.

\bibitem{ChP} Chari, V., Pressley, A.: {\em A guide to quantum groups}, Cambridge University Press, Cambridge 1994.

\bibitem{KT1} Khoroshkin, S. M., and Tolstoy,  V. N.: {\em Extremal projector and universal R-matrix for quantized contragredient Lie (super)algebras}, Quantum groups and related topics (Wroclaw,
1991), 23--32, Math. Phys. Stud., 13, Kluwer Acad. Publ., Dordrecht, 1992.

\bibitem{Is}  Isaev, A.: {\em Lectures on Quantum Groups and Yang-Baxter Equations}, arXiv:2206.08902.


\bibitem{Zh}Zhang, R.B.: {\em Universal L operator and invariants of the quantum supergroup $U_q (\g\l(m/n))$}, J. Math. Phys., {\bf33} (1992), 1970--1979.


\bibitem{DGL} Dancer, K.A.--Gould, M.D. and Links, J.: {\em Lax Operator for the Quantised Orthosymplectic Superalgebra $U_q [\o\s\p (m|n)]$},  Algebr. Represent. Theory {\bf 10} (2007) 593==617.

\bibitem{AlgMS1} D. Algethami, A. Mudrov, V. Stukopin: {\em Solutions to graded reflection equation of GL-type}, Lett. Math. Phys., {\bf 114}(2024), 22.

\bibitem{KT} Khoroshkin, S. M., and Tolstoy,  V. N.: {\em Universal R-Matrix for Quantized (Super)Algebras}, Commun. Math. Phys., {\bf 141} (1991),  599--617.




\bibitem{KSkl} Kulish, P. P., Sklyanin, E. K.: {\em Algebraic structure related to the reflection equation}, J. Phys. A, {\bf 25} (1992), 5963--5975.

\bibitem{KSS} Kulish, P. P., Sasaki, R., Schwiebert, C.: {\em Constant Solutions of Reflection Equations and Quantum Groups}, J.Math.Phys., J.Math.Phys., {\bf 34} (1993), 286--304.

\bibitem{NS}Noumi, M. and Sugitani, T.: {\em Quantum symmetric spaces and related q-orthogonal polynomials}, Group Theoretical Methods in Physics (ICGTMP),  World Sci. Publ.,
    River Edge, NJ,
  (1995), 28--40.

\bibitem{NDS} Noumi, M., Dijkhuizen,  M.S., and Sugitani, T.: {\em Multivariable Askey-Wilson polynomials
and quantum complex Grassmannians}, AMS Fields Inst. Commun. {\bf 14} (1997), 167--177.














\bibitem{AACLFR} Arnaudon, D., Avan, J., Crampe, J., Doikou, A., Frappat, L., Ragoucy, E.: {\em General boundary conditions for the $\s\l(N)$ and $\s\l(M|N)$ open spin chains}, J. Stat. Mech. {\bf 0408}  (2004),  P08005.

\bibitem{AACLFR1} Arnaudon, D., Avan, J., Crampe, J., Doikou, A., Frappat, L., Ragoucy, E.: {\em Bethe ansatz equations and exact S matrices for the $\mathfrak {o}\s\p(M|2n)$ open super-spin chain}, Nucl. Phys., B {\bf 687} \#3, (2004) 257--278


\bibitem{L} Lima-Santos, A.: {\em Reflection matrices for the $U_q[sl(m|n)^{(1)}]$ vertex model}, J. Stat. Mech., {\bf 0908}, P08006 (2009).


\bibitem{DK}Doikou, A., Karaiskos, N.: {\em New reflection matrices for the $U_q(gl(m|n))$ case}, J. Stat. Mech., {\bf 0909}, L09004 (2009).

\bibitem{BW}  Bao H., Wang, W.: {\em A new approach to Kazhdan-Lusztig theory of type B via quantum
symmetric pairs}, Soci$\acute{e}$t$\acute{e}$ math$\acute{e}$matique de France (2018).

\bibitem{FM} Freidel, L., Maillet, J. M.: {\em Quadratic algebras and integrable systems}, Phys. Lett. B, {\bf 262}  \# 2,3,  (1991), 278-- 284.






\bibitem{DH} Dieck, T. and H\"{a}ring-Oldenburg, R.: {\em Quantum groups and cylinder braiding}, Forum
Math. {\bf 10} no. 5 (1998), 619--639.

\end{thebibliography}
\end{document}